\documentclass[12pt]{amsart}
\usepackage{amsthm}
\usepackage{amsmath}
\usepackage[margin=1in]{geometry}
\usepackage{amssymb,verbatim,enumerate,enumitem, color}
\usepackage{tikz}
\usepackage{tikz-cd}
\usepackage{multirow}
\usepackage{longtable}
\usepackage[skip=2pt]{caption}

\theoremstyle{plain}
\newtheorem{theorem}{Theorem}[section]          
\newtheorem{corollary}[theorem]{Corollary}     
\newtheorem{lemma}[theorem]{Lemma}              
\newtheorem{proposition}[theorem]{Proposition}

\theoremstyle{definition}
\newtheorem{definition}[theorem]{Definition}

\newtheorem{remark}[theorem]{Remark}
\newtheorem{example}[theorem]{Example}

\newtheorem{maintheorem}{Theorem}

\numberwithin{table}{section}

\newcommand{\Z}{\mathbb{Z}}
\newcommand{\Q}{\mathbb{Q}}
\newcommand{\R}{\mathbb{R}}\newcommand{\C}{\mathbb{C}}
\newcommand{\HH}{\mathbb{H}}

\newcommand{\g}{\mathfrak{g}}

\DeclareMathOperator{\im}{Im}

\DeclareMathOperator{\diag}{diag}
\DeclareMathOperator{\curv}{curv}
\DeclareMathOperator{\rk}{rk}

\newcommand{\tensor}{\otimes}

\newcommand{\Isom}{\mathrm{Isom}}
\newcommand{\comp}{b_0(\Sigma)}

\title{Positive curvature and rational ellipticity in cohomogeneity three}
\author{Elahe Khalili Samani}
\address[E. Khalili Samani]{Department of Mathematics,\newline\indent Clark University,\newline\indent Worcester, MA 01610, USA}
\email{ekhalilisamani@clarku.edu}
\author{Marco Radeschi}
\address[M. Radeschi]{Universit\'a Degli Studi Di Torino,\newline\indent Dipartimento Di Matematica "G. Peano",\newline\indent Via Carlo Alberto 10, 1023 Torino (TO), Italy}
\email{marco.radeschi@unito.it}

\begin{document}

\begin{abstract}
We prove that a closed, simply connected, positively curved, cohomogeneity-three manifold whose quotient space has no boundary is rationally elliptic,
thus providing a generalization of similar results regarding rational ellipticity of homogeneous, cohomogeneity-one, and almost non-negatively curved cohomogeneity-two manifolds.
\end{abstract}

\maketitle

\section{Introduction}

A long-standing conjecture due to Bott, Grove, and Halperin states that a closed, simply connected, non-negatively curved manifold $M$ is rationally elliptic, that is, the total rational homotopy
$\bigoplus_{i\geq 2}\pi_i(M)\tensor\Q$ is finite dimensional. On the other hand, the problem of exploring the extent to which an effective group action restricts the topology
of the underlying manifold has been of special interest in both transformation groups and the study of non-negatively curved manifolds. 
These have motivated a series of results about rational ellipticity of Riemannian manifolds admitting isometric group actions. It was proved by Serre \cite{Ser53}
that compact, simply connected homogeneous spaces are rationally elliptic. This result was generalized to simply connected cohomogeneity-one manifolds by Grove and Halperin \cite{GH87},
and later by Grove and Ziller \cite{GZ12} to manifolds admitting polar actions with either flat or spherical sections. In \cite{GWY19}, Grove, Wilking, and Yeager proved that a closed, simply connected almost non-negatively curved manifold which admits a cohomogeneity-two action by a compact, connected Lie group is rationally elliptic. In \cite{KR24}, the authors introduced a new criterion for rational ellipticity of $G$-manifolds which led to generalizing the result of Grove-Wilking-Yeager to non-negatively curved, cohomogeneity-three manifolds under the additional assumption of the action being infinitesimally polar.

The current paper deals with more general cohomogeneity-three actions on positively-curved manifolds:

\begin{maintheorem}\label{main-thm:cohomogeneity three}
Suppose $M$ is a closed, simply connected, positively curved Riemannian manifold which admits a cohomogeneity-three action by a closed Lie group $G$ with $\partial (M/G)=\emptyset$. 
Then $M$ is rationally elliptic. 
\end{maintheorem}

\begin{remark}
The case in which the quotient space $M/G$ is an orbifold follows from Theorem C in \cite{KR24}. Therefore, throughout the paper, we assume that $M/G$ has at least one non-orbifold point.
\end{remark}

The proof of Theorem \ref{main-thm:cohomogeneity three} goes through an analysis of the quotient space. In the first part, we develop two methods to recognize the rational ellipticity 
of the underlying manifold when the quotient space has a certain structure: the first one applies tools from transformation groups and the theory of positively curved manifolds 
to restrict the group $G$ when the quotient $M/G$ has a one-dimensional singular stratum of weight two. When combined with rational homotopy theory, 
this restriction then forces $M$ to have a nice rational homotopy type. The second method uses two special subspaces in $M/G$ to provide a double disk bundle decomposition for $M$.
In the second part of the paper we analyze the possible singular stratifications of the quotient showing in particular that all these cases can be covered by the results of the first part.

This paper is organized as follows. In Section \ref{S:preliminaries}, we introduce some notations and collect the tools required throughout the paper. Sections \ref{S:homotopy}
and \ref{S:disk bundle} cover the techniques applied in the proof of Theorem \ref{main-thm:cohomogeneity three}. Section \ref{S:singular strata} is devoted to classifying the singular stratification 
of the quotient space. In Section \ref{S:proof}, we apply the techniques of Sections \ref{S:homotopy} and \ref{S:disk bundle} to prove Theorem \ref{main-thm:cohomogeneity three}. 
Finally, Appendix A provides a list of actions which lead to singular stratifications discussed in Section \ref{S:singular strata}.

\subsection*{Acknowledgements}
The authors wish to thank Alexander Lytchak and Daryl Cooper for discussions and suggestions on how to prove Lemma \ref{L:hom-to-diff}.
The second author would like to thank both the organizers of the Trimester Program ``Metric Analysis'' and the Hausdorff Center of Mathematics in Bonn where the trimester took place, for the excellent working condition when this work was concluded.
\section{Preliminaries}\label{S:preliminaries}

\subsection{Notation on isometric group actions}
Let $G$ be a compact Lie group which acts isometrically on a Riemannian manifold $M$. In this paper, we will use the following notation:
\begin{itemize}
\item  The action of $g$ on $p\in M$ will be denoted by $g\cdot p$.
\item Given a point $p\in M$, the {orbit through $p$} will be denoted by $G\cdot p$ or $L_p$. The {isotropy group} at $p$ will be denoted by $G_p$.
\item Given an orbit $L$, we denote by $\nu L$ the normal bundle of $L$. Moreover, the normal space of $L$ at a point $p\in L$ will be denoted by $\nu_pL$. If the Euclidean metric on $\nu L$ 
is understood, we denote by $\nu^{<\epsilon} L=\{v\in \nu L\mid \|v\|<\epsilon\}$ the $\epsilon$-disk bundle.
\item The quotient $M/G$ will be understood as a metric space, with the metric induced from the $G$-invariant Riemannian metric on $M$. We will denote by $\pi:M\to M/G$ the canonical projection.
\end{itemize}

\subsection{Alexandrov spaces with positive curvature}

Fix $\kappa\in\R$, and let the $\kappa$-comparison space $M_\kappa$ denote the unique simply connected, $2$-dimensional space with constant curvature $\kappa$. 
Given a length space $X$ and a geodesic triangle $\Delta$ in $X$ with vertices $x,y,z$, define the $\kappa$-\emph{comparison triangle} $\tilde{\Delta}$ (whenever it is defined) 
as a triangle in $M_\kappa$ with vertices $\tilde{x}, \tilde{y}, \tilde{z}$ and edges of the same length as the corresponding ones in $\Delta$. Furthermore, we define the 
$\kappa$-\emph{comparison angle} $\tilde\measuredangle_\kappa x^y_z$ as the angle of $\tilde{\Delta}$ at the vertex $\tilde{x}$.

A compact length space $X$ is an \emph{Alexandrov space with $\curv\geq \kappa$} if for every $x\in X$, there is a neighborhood $U_x$ of $x$ such that for any three points $x_0, x_1, x_2$ in $U_x$, and geodesic triangles
$\Delta_i$ with vertices $x$, $x_i$, $x_{i+1}$ (where the indices are taken modulo 3), one has
\[
\tilde\measuredangle_\kappa x^{x_0}_{x_1}+\tilde\measuredangle_\kappa x^{x_1}_{x_2}+\tilde\measuredangle_\kappa x^{x_2}_{x_0}\leq 2\pi
\]
whenever all three angles are defined.

\begin{example}
The examples we will mostly care about come from quotients of group actions, that is, if $M$ is a Riemannian manifold with $\sec_M\geq \kappa$, and $G$ is a Lie group acting properly on $M$ 
by isometries, then the quotient space $M/G$ is an Alexandrov space with $\curv\geq \kappa$.
\end{example}

Fixing an Alexandrov space $X$ with $\curv\geq \kappa$ and a point $x\in X$, one defines the \emph{space of directions} $\Sigma_x$ of $X$ at $x$ as follows: First take the set $\Sigma_x'$ 
of unit length geodesics $\{\gamma_\alpha:[0, \epsilon) \to X\}_\alpha$ starting from $x$. Next, define a pseudo-distance on $\Sigma'_x$ by
\[
\measuredangle(\gamma_\alpha,\gamma_\beta)=\lim_{s,t\to 0}\tilde\measuredangle_\kappa x^{\gamma_\alpha(s)}_{\gamma_\beta(t)}.
\]
This defines a distance on $\Sigma'_x/\!\!\sim$ where $\gamma_1\sim \gamma_2$ if $\measuredangle_\kappa x^{\gamma_1(t)}_{\gamma_2(t)}=0$. 
Finally, define $\Sigma_x$ as the completion of $\Sigma_x'/\!\!\sim$. One has the following results:

\begin{proposition}
Let $(X,d)$ be an Alexandrov space and $x\in X$. Then:
\begin{itemize}
\item \cite[Theorem 10.9.3 and Corollary 10.9.5]{BBI01} For any sequence $\lambda_i\to \infty$, the sequence of pointed metric spaces $(X, \lambda_i d, x)$ converges in the pointed Gromov-Hausdorff sense to the Euclidean cone over $\Sigma_x$, which is  an Alexandrov space with $\curv\geq 0$. This cone is called the \emph{tangent cone at $x$}, and is denoted by $T_xX$.
\item \cite[Corollary 10.9.6]{BBI01} The space of directions $\Sigma_x$ is a compact Alexandrov space with $\curv\geq 1$ (if $\dim X>2$).
\item \cite[Prop. 1.8]{Gro02} If $X=M/G$, then $\Sigma_x=\Sigma'_x$ (that is, every direction exponentiates to a geodesic) and it is isometric to $\nu^1(G\cdot p)/G_p$,
where $\pi:M\to X$ is the canonical projection and $p\in \pi^{-1}(x)$.
\end{itemize}
\end{proposition}

\begin{definition}\label{D:gamma-pm}
For a piecewise geodesic $\gamma:(-\epsilon,\epsilon)\to X$ with breaking point at $\gamma(0)=x$, the two geodesics $\gamma_1,\gamma_2:[0,\epsilon)\to X$ given by $\gamma_1(t)=\gamma(t)$ 
and $\gamma_2(t)=\gamma(-t)$ represent points in $\Sigma_x$ which we will denote by $\gamma^+(0)$ and $\gamma^-(0)$, respectively.
\end{definition}

\begin{definition}
Given an Alexandrov space $X$ and a positive integer $k$, the \emph{$k$-dimensional stratum of $X$} is the set of points $x\in X$ such that $T_xX$ is isometric to a product $\R^k\times X'$ 
for some space $X'$, but not isometric to $\R^{k+1}\times X''$ for any $X''$.
\end{definition}

It is well-known that each $k$-dimensional stratum is a smooth Riemannian manifold, locally convex in $X$, and its closure is the union of the stratum itself, with strata of lower dimension. 
In particular, the zero-dimensional stratum consists of a discrete set of points, and the $1$-dimensional stratum consists of a union of geodesics in $X$, ending at zero-dimensional strata.

\subsection{Classification of singular points in the quotient of a cohomogeneity-three group action}\label{SS:points}

If $G$ is a connected Lie group which acts properly and isometrically on a closed, simply connected, positively curved Riemannian manifold $M$ with $\partial (M/G)=\emptyset$, then we know the following:
\begin{enumerate}
\item $M/G$ is an Alexandrov space with $\curv>0$ and the canonical projection $\pi:M\to M/G$ is a submetry (i.e. $\pi(B_r(p))=B_r(\pi(p))$ for any $p\in M$ and $r>0$). 
\item By Corollary IV.4.7 in \cite{Br72}, it follows that $M/G$ is homeomorphic to $S^3$.
\item Singular strata are either $1$-dimensional or zero-dimensional, and the whole singular set is a disjoint union of circles and graphs.
\end{enumerate}

We can in fact say much more about the singular set: recall that given a $G$-orbit $L$ of $M$ and a point $p\in L$, by the Slice Theorem, the normal exponential map 
$U_p:=\nu_p^{<\epsilon}L\to M$ induces a strata-preserving homeomorphism $U_p/G_p\to B_\epsilon(p_*)\subseteq M/G$ where $p_*=\pi(p)$.
Furthermore the tangent cone $T_{p_*}M/G$ is isometric to $\nu_pL/G_p$, where the slice representation of $G_p$ on $\nu_pL$ has cohomogeneity three and $\nu_p L/G_p$ 
has empty boundary. The following classification can be found in Straume's classification \cite{Str94} but it is not immediate to parse it from there, so for the sake of completeness,
we provide a proof here.

\begin{lemma}\label{L:codim2}
Suppose that a (not necessarily connected) compact Lie group $H$ acts linearly on a Euclidean space $V$ by isometries, with $\dim V/H=3$ and $\partial(V/H)=\emptyset$. 
Let $S(V)$ be the unit sphere in $V$. Then $S(V)/H$ is a $2$-orbifold, isometric to:
\begin{enumerate}
\item $S^2_{pp}=S^2(1)/\Z_p$ (resp. $S^2_{pqr}(1)=S^2(1)/G_{pqr}$, where the possible integers $p,q,r$ and the finite group $G_{pqr}\subseteq\mathrm{SO}(3)$ are given in Table \ref{Ta:groups}). 
In this case, $V\simeq \R^3$ and $H=\mathbb{Z}_p$( resp. $H=G_{pqr}$) with the action provided in Table  \ref{Ta:groups}.
\begin{table}[h!]\label{Tb:Gpqr}
\centering
\begin{tabular}{|c|c |c |c | c|} 
\hline
Case& $p$ & $q$ & $r$ & $G_{pqr}$\\\hline\hline
a.& $p$ & - & - & \{Rotation by $2\pi \over p$ on the $xy$-plane\}\\\hline
b.& $2$ & $2$ & $k$ & \{Isometries of regular $k$-gon in $\R^2\subseteq \R^3$\}$\cap\mathrm{SO}(3)$\\\hline
c.& $2$ & $3$ & $3$ & \{Isometries of the tetrahedron\}$\cap\mathrm{SO}(3)$\\\hline
d.& $2$ & $3$ & $4$ & \{Isometries of the cube and octahedron\}$\cap\mathrm{SO}(3)$\\\hline
e.& $2$ & $3$ & $5$ &\{Isometries of the dodecahedron and icosahedron\}$\cap\mathrm{SO}(3)$\\
\hline
\end{tabular}
\smallskip
\caption{Finite subgroups of $\mathrm{SO}(3)$.}
\label{Ta:groups}
\end{table}
\item $S^2_{pq}=S^3/S^1$ with $\gcd(p,q)=1$, and $H=S^1$ acts on $S^3$ via the weighted Hopf action.
\item $S^2_{pp}(4)=S^2(4)/\Z_p$ (resp. $S^2_{pqr}(4)=S^2(4)/G_{pqr}$ where $p,q,r$ and $G_{pqr}$ are as in Table  \ref{Ta:groups}). In this case the group $H$ admits an extension $1\to S^1\to H\to K\to 1$, where $S^1=H_0$ acts on $V\simeq \C^2$ via the Hopf action, and $K=\Z_p$ (resp. $K=G_{pqr}$).
\item $S^2_{pq}$ with $d=\gcd(p,q)>1$. In this case, $V\simeq \R^4$ and $H=S^1\times \Z_d$ maps into a maximal torus $T^2\subset\mathrm{SO}(4)$ 
as $(s,m)\mapsto \left({p\over d}s+m,{q\over d}s\right)$.
\end{enumerate}
\end{lemma}

\begin{proof}
Let $H_0$ be the identity component of $H$. Since we are assuming that $\partial(V/H)=\emptyset$, it follows that $\partial(S(V)/H_0)=\emptyset$ as well, and the principal isotropy group is trivial by \cite[Lemma 5.1]{Wil04}. In particular, exceptional orbits have finite isotropy group.

In our case, $S(V)/H_0$ is 2-dimensional, hence an orbifold (cf. \cite[Corollary 1.2]{LT10}) and therefore the action is infinitesimally polar (cf. \cite[Theorem 1.1]{LT10}). By \cite[Theorem 1.6]{Lyt10}, we have that there are no singular orbits, and therefore all isotropy groups for the $H_0$-action (hence for the $H$-action as well) are finite.

If $H_0$ is trivial, then $H$ is finite, and thus $\dim V=3$. The specific groups follow from classical classification results on finite subgroups of $\mathrm{SO}(3)$, giving case 1.

If $H_0\neq \{e\}$, we know that $\rk H_0=1$ for otherwise a maximal torus of dimension $\geq 2$ would have singular points. Thus $H_0$ is either $S^1$ or it is covered by $\mathrm{SU}(2)$. 
However, by standard representation theory of $\mathrm{SU}(2)$, almost free $\mathrm{SU}(2)$ representations have $\dim S(V)=4k-1$, thus $\dim S(V)/H_0=4k-4\neq 2$. 
The only possibility is then $H_0=S^1$, $V=\C\oplus \C$, $S(V)=S^3$, and the action is $w\cdot (z_1,z_2)=(w^pz_1,w^qz_2)$ for some $p,q$ with $\gcd(p,q)=1$, and the quotient space is $S^3/H_0=S^2(4)$ 
if $p=q=1$, $S^2_{1,q}$ if $q>1$, or $S^2_{p,q}$ if $q,p>1$. If $H=H_0$, this gives case 2.

Finally, if $H$ is disconnected and not discrete, then it fits in the exact sequence $1\to S^1\to H\to \pi_0(H)\to 1$ and the $H$-action on $S^3$ descends to a $\pi_0(H)$-action by isometries on $S^3/H_0$, 
which can be individually lifted to isometries of $S^3$. On the one hand, it is known that the subgroup $\mathrm{O}(4)^{desc}\supseteq H$ of isometries of $S^3\subset \C\oplus \C$ 
that descend to $S^2_{p,q}=S^3/H_0$ is the normalizer of $H_0\subset O(4)$, which is either the maximal torus $T^2$ of $\mathrm{O}(4)$ if $(p,q)\neq(1,1)$ 
or $\mathrm{U}(2)\cup\mathrm{U}(2)\circ c$ (where $c(z_1,z_2)=(\bar{z}_1,\bar{z}_2)$) if $p=q=1$. On the other hand, $c\notin H$ because otherwise $S^3/H$ would be covered by $S^3/c\cdot H_0\simeq D^2$ 
hence it would have nonempty boundary. Therefore either way, the group $\pi_0(H)$ is contained in a connected subgroup of $\mathrm{O}(4)^{desc}/H_0=Isom(S^3/H_0)$. 
Furthermore, by \cite{Men21} all isometries in the identity component of $\Isom(S^3/H_0)$ can indeed be lifted. Since every finite subgroup of $\Isom(S^2(4))_0=SU(2)$ 
(or $\Isom(S^2_{ab})_0=S^1$) provides a unique $H$-action on $S^3$ which restricts to the corresponding action of $H_0$, one gets to case 3 if $S^3/H_0=S^2(4)$, 
and to case 4 if $S^3/H_0=S^2_{pq}$ (in case 4, the lifted action of a generator of $\pi_0(H)=\Z_d$ commutes with the $H_0$-action by \cite{Men21}).
\end{proof}

\begin{definition}
We will call a singular point $p_*\in M/G$ a \emph{point of type (1), (2), (3), (4)} if the isotropy representation of $H=G_p$ on $\nu_pL_p$ ($p\in \pi^{-1}(p_*)$) fits in the corresponding class of Lemma \ref{L:codim2} above. We furthermore divide the points of type (1) into types (1.a) through (1.e) according to the specific group $G_{pqr}$ from Table \ref{Ta:groups}.
\end{definition}

\begin{definition}
We call a singular point $p_*$ of $M/G$:
\begin{enumerate}
\item an \emph{orbifold point} if there is a neighbourhood of $p_*$ isometric to a Riemannian orbifold. By the Slice Theorem, these correspond to points of type (1).
\item a \emph{small point} if for any three points in $\nu_p^1L_p/G_p$, the sum of their distances is $\leq \pi$ (in the literature, such a space is said to have $3$-extent $\leq \pi/3$). By \cite{HK89} and \cite{MGS22} these correspond to either non-orbifold points, or points of type (1.c), (1.d), (1.e).
\end{enumerate}
\end{definition}

 {As a result, we have in particular the following:

\begin{lemma}\label{lem:zero-dimensional}
Suppose that $X$ is an Alexandrov space without boundary that arises as the quotient space of a cohomogeneity-three action on a closed, simply connected, 
positively curved Riemannian manifold $M$. Then:

\begin{enumerate}
\item The singular set $\Gamma$ is a disjoint union of a (possibly disconnected) graph in $X$ whose vertices are zero-dimensional strata and whose edges are one-dimensional strata and, possibly, circles corresponding to closed one-dimensional strata.
\item Every vertex of $\Gamma$ has valency at most three.
\item $\Gamma$ contains at most three small points.
\item $\Gamma$ contains at most three vertices with valency $\leq 2$.
\end{enumerate}
\end{lemma}

\begin{proof}
Given a point $x_*\in X$ and $x\in \pi^{-1}(x_*)$ with respect to the canonical projection $\pi:M\to M/G$, the Slice Theorem implies that a neighborhood of $x_*$ is equivalent to the cone 
over $S(V)/H$, where $V=\nu_p(G\cdot x)$ and $H=G_x$. By Lemma \ref{L:codim2} above we have that:
\begin{itemize}
\item If $S(V)/H$ is of type (1.a), then $x_*$ is a point in the interior of an edge of the singular stratum.
\item If $S(V)/H$ is of type (1.b), (1.c), (1.d), (1.e) or (3) with $G_x=G_{pqr}$, then $x_*$ is a vertex of the singular stratum, adjacent to three edges.
\item If $S(V)/H$ is of type (2) or (4) with $p,q>1$, or of type (3) with $G_x=\Z_p$ then $x_*$ is a singular vertex, adjacent to two edges.
\item If $S(V)/H$ is of type (2) with $1=p<q$, then $x_*$ is a singular vertex adjacent to one edge.
\item If $S(V)/H$ is of type (2), with $p=q=1$, then $x_*$ is an isolated singular vertex.
\end{itemize}

By \cite{HK89, GM}, since $X$ has  positive curvature, there can be at most three small points. Since points with valency $\leq 2$ are small, 
it follows in particular that there are at most three points with valency $\leq 2$.
\end{proof}

\begin{definition}\label{D:lab-graph}
We define a \emph{labeled graph} as a graph with edges weighted by integers $\geq 2$, and vertices marked either as \emph{hollow}, or \emph{full}.

Given a $3$-dimensional quotient $X=M/G$ with $\partial X=\emptyset$, the \emph{labeled graph} of $X$ is the labeled graph $\Gamma$ with:
\begin{itemize}
\item A vertex for every zero-dimensional stratum, marked as hollow if it is an orbifold point and full otherwise.
\item One edge for each one-dimensional stratum of $X$, and multiplicity given by the integer $p$ such that the space of directions of the points in that stratum is $S^2_{pp}$.
\end{itemize}
\end{definition}
}

\subsection{Quotient geodesics}
 
By the Slice Theorem, it follows that for any point $p_*$ in a quotient space $M/G$, the space of directions $\Sigma_{p_*}\simeq \nu^1_p(G\cdot p)/G_p$ inherits an involution 
$a:\Sigma_{p_*}\to \Sigma_{p_*}$ from the antipodal map on $\nu_p^1(G\cdot p)$. A piecewise geodesic curve $\gamma:[a,b]\to M/G$ is called a \emph{quotient geodesic} 
if at any point $t$, the two one-sided derivatives $\gamma^+(t),\gamma^-(t)\in \Sigma_{\gamma(t)}$ (see Definition \ref{D:gamma-pm}) are related by $\gamma^-(t)=a(\gamma^+(t))$. 
It was shown in \cite{MR20} that quotient geodesics coincide with projections to $M/G$ of horizontal geodesics in $M$. Furthermore, by \cite{LT10}, it follows that if a quotient geodesic
starts tangent to the $1$-dimensional stratum of $M/G$, then it stays in the singular stratum for all times. The image of such a quotient geodesic will then be a union of edges of the singular set 
$\Gamma$ (see the previous Section).

\begin{lemma}\label{L:no-only-orb}
Let $G$ be a connected, compact Lie group which acts isometrically on a closed, positively curved Riemannian manifold $M$ with cohomogeneity three and $\partial(M/G)=\emptyset$. 
Then any quotient geodesic $c:\R\to M/G$ in the singular set $\Gamma$ must intersect a non-orbifold point. In particular, $\Gamma$ only consists of a (possibly disconnected) graph.
\end{lemma}

\begin{proof}
Suppose by contradiction that $c:\R\to M$ is a horizontal geodesic projecting to the singular set and only intersecting orbifold points of $M/G$. We can assume that $p:=c(0)$ projects via $\pi:M\to M/G$ to a point $p_*$ in a singular edge of $M/G$ (hence $G_p$ is cyclic).
Let $M'$ be the connected component of $M^{G_p}$ through $p$ and $G'=(N(G_p)/G_p)_0$. Then:
\begin{itemize}
\item $M'$ contains $c$, and $G'$ acts on $M'$ with $M'/G'$ contained in the singular set and containing $\pi\circ c$. In particular, $\dim M'/G'=1$ and thus $G'$ acts on $M'$ with cohomogeneity one.
Moreover, $G'_p=\{e\}$ and, in particular, the principal $G'$-orbits are diffeomorphic to $G'$.
\item Since $G$ is connected and $G_p$ is cyclic by Lemma \ref{L:codim2}, we can assume that $G_p$ is contained in the maximal torus $T$ of $G$, hence $T/G_p\subseteq N(G_p)/G_p$, and thus the principal $G'$-orbits have dimension $\geq 1$. Therefore, $\dim M'\geq 2$ 
and since $M'$ is totally geodesic in $M$, we have $\sec_{M'}>0$.
\item Since $\pi\circ c$ only meets orbifold points, it follows that $G_{c(t)}$ is finite for all $t$, and in particular, so is $G'_c(t)=(G_{c(t)}\cap N(G_p))/G_p$.
\end{itemize}
Summing up, $M'$ is a compact, positively curved manifold with a cohomogeneity-one $G'$-action ($G'$ connected), all of whose isotropy groups are finite and thus all orbits 
have the same dimension. However, this action lifts to a cohomogeneity-one action on the (compact) universal cover  $M''$ of $M'$ and by Lytchak \cite{Lyt10}, there must be singular orbits in $M''$. Since the orbits on $M''$ have the same dimension as the corresponding orbits as in $M'$, this implies the existence of singular points on $M'$ as well, contradicting the equidimensionality of orbits in $M'$.
\end{proof}

\subsection{Branched covers of Alexandrov spaces along singular cycles}\label{SS:branched cover}

One of the techniques that is applied to restrict the structure of the singular strata of the quotient space, is to pass to the $k$-fold branched cover of the quotient along a singular cycle.

\begin{definition}
Let $X$ be a metric space homeomorphic to $S^3$, and let $c$ be a simple closed curve in $X$. The \emph{$k$-fold branched cover of $X$, branched over $c$}, denoted $X_k(c)$, 
is a metric space obtained by the following steps:
\begin{enumerate}
\item Take the cover $b':\widetilde{X\setminus c}\to X\setminus c$ corresponding to the map
\[
\pi_1(X\setminus c)\to H_1(X\setminus c)\cong \Z\to \Z_k
\]
\item Define the length $\ell(\gamma)$ of a curve $\gamma:[0,1]\to \widetilde{X\setminus c}$ as the length of $b'\circ \gamma$, and endow $\widetilde{X\setminus c}$ with the metric induced by $\ell$.
\item Define $X_k(c)$ as the metric completion of $\widetilde{X\setminus c}$, and define $b_X:X_k(c)\to X$ as the unique continuous extension of $b'$.
\end{enumerate}
\end{definition}

The definition of the space $X_k(c)$ and its map $b_X$ coincide with the more topological definitions common in knot theory \cite{BZ03} but it also adds information about the metric space structure.
We now define branched covers of graphs:

\begin{definition}
Fix a natural number $k>1$. Given a graph $\Gamma$ with weighted edges, and a simple cycle $c$ in $\Gamma$ made of edges with weight $\geq k$, define the \emph{$k$-fold cover of $\Gamma$ branched along $c$}, written ${\Gamma}_k(c)$, as the graph constructed as follows:
\begin{enumerate}
\item Letting $V(\Gamma)=V_1\cup V_2$ be the set of vertices of $\Gamma$, where $V_1$ denotes the vertices contained in $c$ and $V_2$ its complement, 
then $V({\Gamma}_k(c)):= V_1\cup (V_2\times \{1,\ldots k\})$
\item For any edge $(u,v)$ of $\Gamma$ with weight $w$, add the following edges to ${\Gamma}_k(c)$:
\begin{itemize}
\item  {If $(u,v)$ is an edge of $c$} with weight $w>k$, add $(u,v)$ to ${\Gamma}_k(c)$ as well, with weight $w/k$.
\item If $u\in V_1$, $v\in V_2$, add $(u, (v,i))$ to ${\Gamma}_k(c)$ for $i=1,\ldots k$,  {with weight $w$}.
\item If $u, v\in V_2$, add $((u,i), (v,i))$ to ${\Gamma}_k(c)$ for $i=1,\ldots k$,  {with weight $w$}.
\end{itemize}
\end{enumerate}
We denote by $b_{\Gamma}:\Gamma_k(c)\to \Gamma$ the branched covering of $\Gamma$, uniquely determined by its action on the vertices of $\Gamma_k(c)$ as:
\begin{itemize}
\item  $b_{\Gamma}|_{V_1}=id_{V_1}$,
\item $b_{\Gamma}|_{V_2\times \{1,\ldots k\}}$ is the projection onto the first factor.
\end{itemize}
\end{definition}

\begin{remark}
Given a $k$-fold branched cover $b_X:X_k(c)\to X$ (resp. $b_{\Gamma}:\Gamma_k(c)\to \Gamma$) we will abuse notation and identify the curve $c$ (resp. the cycle $c$) with its preimage $b_X^{-1}(c)$ (resp. $b_{\Gamma}^{-1}(c)$).
\end{remark}

\begin{proposition}\label{P:double-branched}
Suppose that $X=M/G$ is an Alexandrov space with $\curv>0$ homeomorphic to $S^3$ and with singular set a graph $\Gamma$. Let $c$ be a simple cycle in $\Gamma$ made of a union of edges, 
and assume that all the edges of $c$ have weight $\geq k$ for some $k$. Then:
\begin{enumerate}
\item the $k$-fold branched cover $b_X:X_k(c)\to X$ is still an Alexandrov space with $\curv>0$.
\item The preimages of non-orbifold points of $\Gamma$ are small points of $X_k(c)$, and the preimage of an edge in $c$ weighted by $w$ consists of points whose space of direction 
is isometric to the spherical suspension of a circle of length $2\pi\over (w/k)$.
\end{enumerate}
\end{proposition}

\begin{proof}
The proof is the same as Lemma 2.3 of \cite{GW14} (only dealing with $k$-fold  instead of two-fold branched covers), but we add it here with some details for the reader's benefit.

All results are clearly true on $X_k(c)\setminus c$, because the branched cover is a local isometry. Furthermore, notice that around the preimage of edges of $c$ with multiplicity $k$ in $X$, 
the space $X_k(c)$ is a manifold with $\sec>0$ and in particular $X_k(c)\setminus b_X^{-1}(\tilde{c})$ is an Alexandrov space with $\curv>0$, where $\tilde{c}$ is the union of edges 
with weights $>k$ in $X$.

Let $p=c(t)$ be a point of $c\subset X$ and let $\bar p=b_X^{-1}(p)$ be the corresponding point of $c\subset X_k(c)$. Since any geodesic from $p$ into $X\setminus c$ has $k$ 
preimages in $X_k(c)$, and the angle between nearby such geodesics is the same as the angle between the corresponding preimages, 
we have that $d_{\bar{p}}b_X:\Sigma_{\bar{p}}\setminus \{c^+(t),c^-(t)\}\to \Sigma_p\setminus\{c^+(t),c^-(t)\}$ is a $k$-fold cover and a local isometry. 
Therefore, $d_{\bar{p}}b_X$ extends to a $k$-fold branched cover $d_{\bar{p}}b_X:\Sigma_{\bar{p}}\to \Sigma_p$, branched over the points $\{c^+(t), c^-(t)\}$. In particular, if $p\in X$ belongs to an edge of $c$ 
with weight $w\in \mathbb{N}$, we have that $\Sigma_p=S^2/\Z_w$ is isometric to the spherical suspension of $S^1$ with diameter $2\pi/w$, hence $\Sigma_{\bar p}$ is isometric to the spherical suspension of $S^1$ with diameter $2k\pi/w=2\pi/(w/k)$. This proves the second statement of (2) but also shows that $X_k(c)\setminus b_{X}^{-1}(\tilde{c})$ is convex. 
In fact, if a minimizing geodesic $\gamma\neq c$ in $X_k(c)$ crosses $b_{X}^{-1}(\tilde{c})$ at a point $\gamma(0)=c(t_0)$, then the first variation formula for the length function implies that the points 
$\gamma(0)^+,\gamma^-(0) \in \Sigma_{c(t_0)}\setminus \{c^+(t_0), c^-(t_0)\}$ would have to lie at distance $\pi$ from each other, contradicting the fact that $\Sigma_{c(t_0)}$ is the spherical suspension 
of a circle of length $<2\pi$. Since the complement of $b_{X}^{-1}(\tilde{c})$ is convex and $b_{X}^{-1}(\tilde{c})$ has empty interior, every geodesic triangle in $X_k(c)$ is a limit 
of geodesics triangles in (the Alexandrov space) $X_k(c)\setminus b_X^{-1}(\tilde{c})$ and thus $X_k(c)$ is itself an Alexandrov space with $\curv>0$, proving (1).

The first statement of (2) stems from the fact that if for some $p=c(t)$ the space of direction $\Sigma_p$ is not an orbifold point, then by Hsiang and Kleiner \cite{HK89} it is \emph{dominated by} 
$S^2(4)$, that is, there is a length non-increasing map $S^2(4)\to \Sigma_p$ (in fact, since the spaces of directions of $\Sigma_p$ at $c^\pm(t)\in \Sigma_p$ is a circle of diameter $\leq {2\pi/k}$, 
it follows that $\Sigma_p$ is dominated by $S^2(4)/\Z_k$). Therefore, $\Sigma_{\bar{p}}$ is dominated by $S^2(4)$, and in particular, it is a small point.
%
\end{proof}

\begin{corollary}
Let $X=M/G$ be a quotient of a closed, simply connected, positively curved Riemannian manifold $M$, with $\partial X=\emptyset$ and $\dim X=3$. Then:
\begin{enumerate}
\item The singular set $\Gamma$ does not contain two disjoint cycles.
\item If $c$ is a simple cycle in $\Gamma$ made of a union of edges with weight $\geq k$ for some $k$, the singular set of $X_k(c)$ is isomorphic, as a weighted graph, to $\Gamma_k(c)$.
\end{enumerate}
\end{corollary}
\begin{proof}
(1) Suppose $c_1$, $c_2$ are disjoint cycles in $\Gamma$. By Lemma \ref{L:no-only-orb}, each of them contains at least one non-orbifold point $p_i\in c_i$. 
By Proposition \ref{P:double-branched}, the double branched cover $X'=X_2(c_1)$ is a positively curved Alexandrov space with at least one singular cycle $c_2'\in b_X^{-1}(c_2)$ and three small points 
$p_1'\in c_1'= b_X^{-1}(c_1)$, $p_2', p_2''\in b_{X}^{-1}(p_2)$. Furthermore, the spaces of directions along $c_2'$ are the same as those along $c_2$, therefore taking the double branched cover along $c_2'$ produces an Alexandrov space with $\curv>0$ and at least four small points, contradicting Hsiang-Kleiner.

(2) By the previous point, no two cycles $c_1, c_2$ of $\Gamma$ are linked (that is, $c_1$ is nullhomotopic in $X\setminus c_2$ and vice versa), and therefore $\Gamma\setminus c$ is contained in a simply connected subset of $X\setminus c$. 
In particular, the preimage of $\Gamma\setminus c$ under the map $b_X:X_k(c)\to X$ is $(\Gamma\setminus c)\times \{1,\ldots k\}$ with the edges preserving the same weight as their image in $X$. On the other hand, the singular points in $c\subset X_k(c)$ are the preimages of those edges of $c\subset X$ with weight $>k$ (cf. the proof of Proposition \ref{P:double-branched}), 
and for each edge $e$ in $\Gamma$ connecting a vertex in $c$ with a vertex in $\Gamma\setminus c$, $b_X^{-1}(e)$ consists of exactly $k$ singular edges. 
By definition, this shows that the singular set of $X_k(c)$ is isomorphic to $\Gamma_k(c)$.
\end{proof}

\part{Techniques for identifying rational ellipticity}\label{P:techniques}

Given a closed, simply connected, positively curved Riemannian manifold $M$ with a cohomogeneity-three action by a group $G$, we apply two different strategies to prove that $M$ 
is rationally elliptic. In some cases, it is possible to compute the cohomology ring of $M$ and to conclude that $M$ is rationally homotopy equivalent to a rationally elliptic space.
In some other cases, $M$ can be decomposed as a double disk bundle over rationally $\Omega$-elliptic spaces. In this section, we discuss the details of these techniques.

\section{Rational homotopy theory}\label{S:homotopy}

\subsection{Restricting possible dimensions and groups in special cases}

Recall that if $M$ is a $G$-manifold with $\sec_M>0$ and $\partial(M/G)=\emptyset$, then the principal isotropy is trivial (cf. \cite[Lemma 5.1]{Wil04}), and by Lemma \ref{L:codim2} every isotropy group $G_p$ has $\rk G_p\leq 1$.
 
\begin{lemma}\label{L:rank2}
Suppose $M$ has $\sec_M>0$ and $G$ acts isometrically on $M$ with cohomogeneity three and $\partial (M/G)=\emptyset$. Then $\rk G\leq 2$.
\end{lemma}

\begin{proof}
Let $T\subset G$ be a maximal torus and assume $\dim T\geq 3$. Then for every $p\in M$, $T_p=G_p\cap T$ has dimension $\leq 1$ since $\rk G_p\leq 1$. 
Let $T'\subseteq T$ denote a 2-torus transverse to all of the (finitely many) $T_p$'s. Then $T'$ acts on $M$ almost freely, which is not possible since $M$ has positive curvature 
(cf.  \cite[Theorem 8.3.5]{Pet16}).
\end{proof}

\begin{proposition}\label{P:groups}
Suppose that $M$ is a cohomogeneity-three $G$-manifold such that $\partial(M/G)=\emptyset$. Assume furthermore that $M/G$ has an edge with multiplicity two.
Then  {up to a finite cover, $G$ is a product of at most two rank-one groups.}
\end{proposition}

\begin{proof}
By Lemma \ref{L:no-only-orb}, such an edge must intersect a non-orbifold vertex $q_*$. Let $p$ be a point projecting to the edge with multiplicity two, let $\gamma(t)$ be a horizontal geodesic 
starting at $p$ and projecting to the edge, and let $q=\gamma(1)$ be a singular point projecting to the non-orbifold point $q_*$. Then $G_p=\Z_2$, and $(G_q)_0=S^1$.

Define $M'=$ the connected component of $M^{G_p}$ through $p$, and $G'=(N(G_p)/G_p)_0$ as in the proof of Lemma \ref {L:no-only-orb}. From that proof, we know that $M'$ is totally geodesic with $\dim M'>1$ and $\sec_{M'}>0$, $G'$ acts on $M'$ with cohomogeneity one, and the map $M'/G'\to M/G$ induced by the inclusion $M'\to M$ has image given by a union of singular edges in $M/G$.

Let now $M''$ be the (compact) universal cover of $M'$, and $G''$ the connected group lifted from $G'$, acting effectively on $M''$ (see \cite[Theorem 9.1 of Chapter I]{Br72}). Notice that the covering $\pi:M''\to M'$ comes with a homomorphism $\rho:G''\to G'$, such that $\pi$ is $\rho$-equivariant. We claim that for every $z''\in M''$ sent to $z'\in M'$, 
the map $\rho|_{G''_{z''}}:G''_{z''}\to G'_{z'}$ is:
\begin{itemize}
\item Injective. In fact, since $\pi$ induces an isometry $B_{\epsilon}(z'')\to B_{\epsilon}(z')$, if $g'',h''\in G''_{z''}$ satisfy $\rho(g'')=\rho(h'')$, then $g''(h'')^{-1}$ acts trivially on $B_{\epsilon}(z'')$ 
and therefore it belongs to the ineffective kernel, which is trivial, implying $g''=h''$.
\item An isomorphism at the level of Lie algebras: this is simply because $G''$ is a cover of $G'$.
\end{itemize}

%
It follows in particular that in our case $H''=\{e\}$, $K_+''=S^1$, and $K_-''$ is either $S^1$ or finite. Furthermore, since $M''$ is simply connected and $G''$ is connected, 
it follows from e.g. Lytchak \cite[Theorem 1.6]{Lyt10} that $K''_\pm$ have positive dimension. Thus, $K''_\pm=S^1$ and Lemma 3.5 in Grove-Wilking-Ziller \cite{GWZ08} implies that $G''$ 
is either $S^1$ or $T^2$. In particular, $G''$ is abelian, and so is $G'$. On the other hand, $G'=(N(G_p)/G_p)_0$ where $G_p$ is a $\Z_2$ subgroup of $G$, 
and therefore $N(G_p)$ must be abelian as well.

Recalling from Lemma \ref{L:rank2} that $\rk G\leq 2$, in order to prove the Proposition it is enough to show that if $\g$ is irreducible of rank two,  then $N(\Z_2)$ is not abelian 
for any $\Z_2$ subgroup of $G$:
\begin{enumerate}
\item If $\g=\mathfrak{su}(3)$ then $G=\mathrm{SU}(3)$ or $\mathrm{PSU}(3)=\mathrm{SU}(3)/\Z_3$. In the first case, a $\Z_2$ subgroup is generated by $\diag(-1,-1,1)$ 
(up to the action of the Weyl group), hence $N(\Z_p)\supseteq S(\mathrm{U}(2)\mathrm{U}(1))$. In the second case, a $\Z_2$ subgroup can be generated by $\diag(1,-1,-1)$ or 
$\diag(\xi,\xi,\xi^{-2})$ where $\xi$ is a sixth root of unity, and the normalizer  contains $S(U(2)U(1))/\Z_3$.
\item If $\g=\mathfrak{so}(5)$ then $G=\mathrm{SO}(5)$ or $\mathrm{Sp}(2)$. In the first case, a $\Z_2$ subgroup is generated by a block diagonal matrix of type $\diag(-I,I,1)$, $\diag(I,-I,1)$
or $\diag(-I,-I,1)$, with normalizer  either containing $\mathrm{SO}(3)$ (in the first two cases) or containing $\mathrm{SO}(4)$.
\item If $\g=\g_2$ then $G=G_2$, which contains $\mathrm{SU}(3)$ as a maximal group with the same torus. Therefore, up to conjugacy, the $\Z_2$ subgroups in $G_2$ 
are generated by the same elements in $\mathrm{SU}(3)$, and thus their normalizers in $G_2$ contain $S(\mathrm{U}(2)\mathrm{U}(1))$.
\end{enumerate}
It then follows that none of the options for rank two Lie groups must occur, and $G$ must be, up to a finite cover, one of $S^1$, $S^3$, $T^2$, $S^1\times S^3$ or $S^3\times S^3$.
\end{proof}

\subsection{Quotients with two non-orbifold points}\label{SS:two}

In this section, we assume that $M$ is a closed, simply connected, positively curved, cohomogeneity-three $G$-manifold with the following properties:

\begin{enumerate}
\item $X:=M/G$ is homeomorphic to $S^3$.
\item There are exactly two non-orbifold points $p_1,p_2\in X$.
\item There is a $1$-dimensional stratum of weight two.
\end{enumerate}

We want to prove that $M$ is rationally elliptic. By the last property and Proposition \ref{P:groups}, it follows that the group $G$ is rationally homotopic to  {$S^1$, $T^2$, 
$S^3$, $S^1\times S^3$, or $S^3\times S^3$. The first two cases have been already discussed by Hsiang-Kleiner \cite{HK89} and Galaz-Garcia-Searle \cite{GGS14}, 
so we will restrict to the groups $S^3$, $S^1\times S^3$, or $S^3\times S^3$}. Again using the fact that the principal isotropy group is trivial, it follows that the manifolds in this section and in Section \ref{SS:one} will have dimension 6, 7, or 9.

\begin{lemma}\label{L:MV}
Let $L_1$ and $L_2$ denote the two singular leaves corresponding to $p_1$ and $p_2$, respectively. If $E_i\to L_i$ denotes the normal sphere bundles, then:
\begin{enumerate}
\item The inclusions $E_i\to M':=M\setminus (L_1\cup L_2)$ induce isomorphisms in rational cohomology, and in particular, there is an isomorphism $\Phi:H^*(E_2)\cong H^*(E_1)$.
\item There is a Mayer-Vietoris sequence in rational cohomology
\[
\ldots \to H^i(M)\stackrel{\iota_1^*\oplus \iota_2^*}{\to} H^i(L_1)\oplus H^i(L_2)\stackrel{\pi_1^*- \Phi\circ\pi_2^*}{\to} H^i(E_1)\to H^{i+1}(M)\to \ldots
\]
where $\pi_i:E_i\to L_i$ denotes the foot-point projection.
\end{enumerate}
\end{lemma}

\begin{proof}
(1) We identify $E_i$ with $\partial B_\epsilon(L_i)$ for some small $\epsilon$, via the normal exponential map. We know that $G$ acts almost freely on $M'$. Moreover, $E_i$ is $G$-invariant 
and the action of $G$ on $E_i$ is almost free as well. Now, consider the commutative diagram
\begin{center}
\begin{tikzcd}
&E_1\arrow[rr] \arrow[dd,"\pi"]\arrow[dr,"\iota_1"]& &E_1\times EG \arrow{dd}[near start]{\hat{\pi}}\arrow[dr,"\iota_1\times id"]\\
&&M' \arrow[rr, crossing over] & &M'\times EG \arrow[dd,"\hat{\pi}"] \\
\{\epsilon\}\times S^2 \arrow[r,"\simeq"]\arrow[dr,hook]&\pi(E_1)\arrow[dr,hook]&&E_1\times_GEG\arrow[ll]\arrow[dr,"\iota_1\times_Gid"]\\
&(0,1)\times S^2\arrow[r,"\simeq"]&\pi(M') \arrow[from=uu,crossing over,near start,"\pi"]&& M'\times_G EG \arrow[ll]
\end{tikzcd}
\end{center}

The rows induce isomorphisms in rational cohomology, while the vertical maps named $\hat{\pi}$ are both principal $G$-bundles. Since the inclusion $\{\epsilon\}\times S^2\to (0,1)\times S^2$ 
is a homotopy equivalence, the commutativity of the diagram implies that the map $\iota_1\times_Gid$ induces isomorphism in rational cohomology. 
Since the map $E_1\times EG\to E_1\times_G EG$ is the pullback of $M'\times EG\to M'\times_GEG$ via $\iota_1\times_Gid$, it follows that the map $\iota_1\times id$ induces an isomorphism in rational cohomology as well. 
Finally, since the inclusion $E_1\to M'$ is homotopic to $E_1\times EG\to M'\times EG$, it must also induce an isomorphism in rational cohomology. The same arguments carry over for $E_2$.

(2) Let $M_i=M\setminus L_{3-i}$, $i=1,2$. Notice that
\[
H^*(M_i, L_i)\cong H^*(M_i, B_\epsilon(L_i))\cong H^*(M', B_\epsilon(L_i)\setminus L_i)\cong H^*(M',E_i)=0,
\]
where the last equality follows from Part (1). In particular, the inclusion $L_i\to M_i$ induces an isomorphism in rational cohomology, and the result follows from the standard Mayer-Vietoris sequence 
for the quadruple $(M, M_1, M_2, M')$.
\end{proof}

\begin{corollary}\label{C:H1=0}
Assume that $M$ is a $G$-manifold which satisfies Properties (1)-(3) at the beginning of this section. Then the singular orbits $L_i$ and their sphere bundles $E_i$ 
have vanishing first rational cohomology. In addition, both $L_i$ and $E_i$ are orientable.
\end{corollary}

\begin{proof}
Let $n=\dim M$, so that $\dim E_i=n-1$ and $\dim L_i=n-4$. From the Mayer-Vietoris sequence above, we have $H^{n-1}(E_i)\cong H^n(M)\cong \mathbb{Q}$, 
which implies that $E_i$ is orientable, and $H^1(E_i)\cong H^{n-2}(E_i)\cong H^{n-1}(M)\cong 0$ (here we used Poincar\'e Duality first, and Mayer-Vietoris second). 
This proves one of the statements. Again, from the Mayer-Vietoris sequence, we have that the sequence $0\to H^1(L_1)\oplus H^1(L_2)\to 0$ is exact, from which $H^1(L_i)=0$ follows.

It remains to prove that $L_i$ is orientable. To do this, we first prove that the normal bundle of $L_i$ is orientable. Fixing a point $p\in L_i$, consider the orbifold bundle 
$\nu^1_p(L_i)\to \nu^1_p(L_i)/G_p\simeq S^2$. Since the quotient is orientable and the sphere $\nu^1_p(L_i)$ is orientable, it follows that $G_p$ acts on $\nu^1_p(L_i)$, 
and on $\nu_p(L_i)$, by orientation-preserving isometries. The Slice Theorem $\nu(L_i)\simeq G\times_{G_p}\nu_p(L_i)$ then implies that $\nu(L_i)$ is orientable as a bundle. 
But since $TM|_{L_i}\simeq TL_i\oplus \nu(L_i)$ and $TM$ is orientable, it follows that $L_i$ is orientable as well.
\end{proof}

To prove the main results of this Section, that is, Propositions \ref{P:two} and \ref{P:one}, we will make multiple use of the following two Lemmas:
 
 \begin{lemma}\label{L:RHT}
Suppose that $Z$ is a simply connected topological space with the same cohomology ring of $S^n$, $\mathbb{C}\mathrm{P}^k$, $S^n\times S^m$ (with $m$ odd) or $S^n\times S^m\times S^r$ (with $n,m$ odd and $r$ even). Then $Z$ has the same rational homotopy type as the corresponding space, and in particular it is rationally elliptic.
\end{lemma}

\begin{proof}
Recall that given a simply connected space $Z$, there is a commutative cochain algebra $A_{PL}(Z)$ with $H^*(A_{PL}(Z))\cong H^*(Z)$ and a free differential graded algebra 
$(\bigwedge V,d)$ with a quasi-isomorphism $\bigwedge V\to A_{PL}(V)$ (that is, a morphism of differential graded algebras, inducing an isomorphism in cohomology) satisfying $d(V)\subset \bigwedge^{\geq 2} V$ (called the \emph{minimal Sullivan algebra} of $Z$). 
The minimal Sullivan algebra is unique up to isomorphism, and it determines the rational homotopy groups of $Z$ since $V_i\cong \operatorname{Hom}(\pi_i(Z),\Q)$. 
In particular, $Z$ is rationally elliptic if and only if $V$ is finite dimensional.

If $H^*(A_{PL}(Z))= H^*(S^n)$ with $n$ odd, let $\alpha\in A_{PL}^n(Z)$ ($d\alpha=0$), be a representative for the generator of $H^n(Z)$. Since $n$ is odd, $\alpha^2=0$ and thus 
$(\bigwedge(\alpha),d=0)$ is the minimal Sullivan algebra of $Z$, hence $V=\operatorname{span}(\alpha)$ is finite dimensional.

If $H^*(A_{PL}(Z))= H^*(S^n)$ with $n$ even, let $\alpha\in A_{PL}^n(Z)$ ($d\alpha=0$), be a representative for the generator of $H^n(Z)$. Since $[\alpha^2]=0\in H^{2n}(Z)$, 
there is a $\beta\in A_{PL}^{2n-1}(Z)$ such that $d\beta=\alpha^2$, and it is easy to see that $(\bigwedge(\alpha, \beta), d\beta=\alpha^2, d\alpha=0)$ is the minimal Sullivan algebra of $Z$. 
In particular, $V=\operatorname{span}(\alpha,\beta)$.

If $H^*(A_{PL}(Z))\cong H^*(\mathbb{C}\mathrm{P}^k)$, let $\alpha\in A_{PL}^2(Z)$ be a cocycle representing the generator of $H^2(Z)$.
Since $\alpha^{k+1}=0$, we have $\alpha^{k+1}=d\beta$ for some $\beta\in A_{PL}^{2k+1}(Z)$. Then $(\bigwedge(\alpha,\beta), d\alpha=0, d\beta=\alpha^{k+1}$) is the minimal Sullivan algebra of $X$, 
hence $V=\operatorname{span}(\alpha,\beta)$.

If $H^*(A_{PL}(Z))= H^*(S^n\times S^m)$ with $n\neq m$ odd, let $\alpha_1, \alpha_2$ be cocycles representing the generators of $H^n(Z)$ and $H^m(Z)$. Then $\alpha_1\wedge \alpha_2$ represents a generator of $H^{n+m}(Z)$ 
and $(\bigwedge(\alpha_1,\alpha_2),d=0)$ is the minimal Sullivan algebra of $X$. In particular, $V=\operatorname{span}(\alpha_1,\alpha_2)$.

If $H^*(A_{PL}(Z))= H^*(S^m\times S^m)$ with $m$ odd, let $\alpha_1, \alpha_2\in A_{PL}^m(Z)$ be cocycles representing a basis of $H^m(Z)$. Then $\alpha_1\wedge \alpha_2$ 
represents a generator of $H^{2m}(Z)$ and $(\bigwedge(\alpha_1,\alpha_2),d=0)$ is the minimal Sullivan algebra of $Z$. In particular, $V=\operatorname{span}(\alpha_1,\alpha_2)$.

If $H^*(A_{PL}(Z))= H^*(S^n\times S^m)$ with $n$ even and $m$ odd, let $\alpha_1, \alpha_2$ be the cocycles representing generators of $H^n(Z)$ and $H^m(Z)$. Furthermore, $\alpha_1^2=d\beta$, 
and $\wedge(\alpha_1, \alpha_2,\beta)$ is the minimal Sullivan algebra of $Z$. Therefore $V=\operatorname{span}(\alpha_1,\alpha_2,\beta)$.

If $H^*(A_{PL}(Z))= H^*(S^n\times S^m\times S^r)$ with $n,m$ odd and $r$ even, let $\alpha_1, \alpha_2, \alpha_3$ be the cocycles representing generators of $H^n(Z)$, $H^m(Z)$, and $H^r(Z)$. Then $\alpha_1^2=\alpha_2^2=0$ , $\alpha_3^2=d\beta$, 
and $\wedge(\alpha_1, \alpha_2,\alpha_3,\beta)$ is the minimal Sullivan algebra of $Z$. Therefore $V=\operatorname{span}(\alpha_1,\alpha_2,\alpha_3,\beta)$.
\end{proof}

The following lemma will be very useful to easily determine the rational cohomology of the singular orbits:
\begin{lemma}\label{L:H0}
Let $M$ be a closed, positively curved manifold with a cohomogeneity-three $G$-action such that $\partial(M/G)=\emptyset$. Let $L_p$ be a singular orbit with isotropy group $H:=G_p$. 
Then $L_p\simeq G/H$ has the same rational cohomology as $G/H_0$.
\end{lemma}

\begin{proof}
Recall from Lemma \ref{L:codim2} that, if $H$ is disconnected, it is either finite or a finite extension $1\to H_0\to H\to K\to 1$, where $K=\pi_0(H)$ acts trivially on $H_0=S^1$.

If $H$ is finite, then we have a fibration $G\to G/H\to BH$ where $\pi_1(BH)\cong H$ acts on $G$ by left translations. Since $G$ is connected, it acts trivially on $H^*(G)$. One can then apply the Serre spectral sequence with rational coefficients and, since $H^{>0}(BH,\Q)=0$, it follows that $H^*(G,\Q)\cong H^*(G/H,\Q)$.

If $H$ is a finite extension $1\to H_0\to H\to \pi_0(H)\to 1$, where $\pi_0(H)$ acts trivially on $H_0=S^1$, then the short exact sequence induces a fibration $BH_0\to BH\to B\pi_0(H)$ where $\pi_1(B\pi_0(H))=\pi_0(H)$ acts trivially on $BH_0$, and therefore the spectral sequence of this fibration implies that the map $BH_0\to BH$ 
induces an isomorphism in rational cohomology. The fibrations
\begin{center}
\begin{tikzcd}
G\arrow[r] \arrow[d]&G \arrow[d]\\
G/H_0\arrow[r] \arrow[d]& G/H \arrow[d]\\
BH_0\arrow[r] &BH
\end{tikzcd}
\end{center}
are pullback of one another. Furthermore, $BH_0$ is simply connected, while $\pi_1(BH)=\pi_0(H)$ acts trivially on $H^*(G)$ since all the elements of $H\subset G$ are isotopic to the identity 
(because $G$ is connected). In particular, we can consider the spectral sequences of the two fibrations in rational cohomology, which are pullback of one another. 
But since $BH_0\to BH$ induces an isomorphism in rational cohomology, it follows that the map $G/H_0\to G/H$ induces an isomorphism in rational cohomology as well.
\end{proof}

\begin{proposition}\label{P:two}
If $M$ is a $G$-manifold which satisfies Properties (1)-(3) at the beginning of Section \ref{SS:two}, then $M$ is rationally elliptic.
\end{proposition}

\begin{proof}
By the discussion at the beginning of this section, $M$ can have dimension $6$, $7$, or $9$. In each dimension, we will determine the rational homotopy type of $M$.

{\textbf{Case 1:}} The $6$-dimensional case. In this case, $G$ is rationally homotopic to $S^3$, the singular leaves are rationally homotopic to $S^2$ by Lemma \ref{L:H0}, and from the spectral sequence 
of $S^3\to E_i\to L_i$ it follows that $E_i$ is rationally $S^3\times S^2$. The Mayer-Vietoris sequence then reads as follows:
\[
0\to H^2(M)\to \Q^2\stackrel{\psi}{\to}\Q\to H^3(M)\to 0,\quad H^4(M)\cong \Q,\quad H^5(M)\cong 0,\quad H^6(M)\cong \Q,
\]
where $\psi=\pi_1^*- \Phi\circ\pi_2^*:H^2(L_1)\oplus H^2(L_2)\to H^2(E_1)$. By Poincar\'e Duality, we have $H^2(M)\cong H^4(M)\cong \Q$. But since $H^2(M)\cong \ker \psi$,
it follows that $\psi$ is surjective and thus $H^3(M)=0$. Therefore, $H^i(M)\cong \Q$ for $i=0,2,4,6$ and $0$ otherwise. The cohomology ring of $M$ is then either that of $\mathbb{C}\mathrm{P}^3$, 
or that of $S^2\times S^4$.  {In the first case, Lemma \ref{L:RHT} gives the result. In the second case, consider the 2-connected principal circle bundle $N\to M$ of $M$. By the Gysin sequence, it follows that $H^*(N)\cong H^*(S^3\times S^4)$, which by Lemma \ref{L:RHT} is rationally elliptic, and thus $M$ is rationally elliptic as well.}

{\textbf{Case 2:}} The $7$-dimensional case. In this case, $G$ is rationally $S^3\times S^1$ and, by Lemma \ref{L:H0}, $L_i$ has the rational cohomology of a quotient $(S^3\times S^1)/S^1$. 
Since $H^1(L_i)=0$ by Corollary \ref{C:H1=0}, it follows that $L_i$ is rationally homotopic to $S^3$. Furthermore, since $E_i$ is the total space of an $S^3$-bundle over $L_i$, 
it has the rational cohomology of $S^3\times S^3$. The Mayer-Vietoris sequence then reads as follows:
\[
H^2(M)\cong 0, 0\to H^3(M)\to \Q^2\stackrel{\psi}{\to} \Q^2\to H^4(M)\to 0,\,H^5(M)\cong H^6(M)\cong 0, H^7(M)\cong \Q.
\]
As before, $\psi$ is the map $\psi:=\pi_1^*- \Phi\circ\pi_2^*: H^3(L_1)\oplus H^3(L_2)\to H^3(E_1)$. It is easy so see that $\pi_1^*\neq 0$, so in particular, $\psi\neq 0$, and $H^3(M)\cong H^4(M)$ 
are isomorphic to either $0$ or $\Q$. The cohomology ring of $H^*(M)$ is then isomorphic to either $H^*(S^7)$ or $H^*(S^3\times S^4)$, respectively. 
 {In both cases, Lemma \ref{L:RHT} proves that $M$ is rationally elliptic.}

{\textbf{Case 3:}} The $9$-dimensional case. In this case, $G$ has the rational cohomology of $S^3\times S^3$ and, by Lemma \ref{L:H0}, $L_i$ has the rational cohomology of $H^*(S^3\times S^3/S^1)\cong H^*(S^3\times S^2)$.
From the spectral sequence of $S^3\to E_i\to L_i$ we then get $H^*(E_i)\cong H^*(S^3\times S^3\times S^2)$. The Mayer-Vietoris sequence then implies the following:
\[
0\to H^2(M)\to \Q^2\stackrel{\psi_2}{\to} \Q\to H^3(M)\to \Q^2\stackrel{\psi_3}{\to} \Q^2\to H^4(M)\to 0,
\]
\[
0\to H^5(M)\to \Q^2\stackrel{\psi_5}{\to} \Q^2\to H^6(M)\to 0,
\]
\[
H^7(M)\cong \Q, H^8(M)\cong 0, H^9(M)\cong \Q,
\]
where $\psi_j=\pi_1^*-\Phi\circ \pi_2^*: H^j(L_1)\oplus H^j(L_2)\to H^j(E_1)$. By Poincar\'e Duality, $H^2(M)\cong H^7(M)\cong \Q$ and thus $\psi_2$ is surjective. 
In particular, the first exact sequence simplifies as
\[
0\to H^3(M)\to \Q^2\stackrel{\psi_3}{\to} \Q^2\to H^4(M)\to 0.
\]
From the exact sequences above and Poincar\'e Duality, we have $H^3(M)\cong H^4(M)\cong H^5(M)\cong H^6(M)$. Furthermore, since the spectral sequence for the fibration 
$S^3\to E_1\stackrel{\pi_1}{\to} L_1$ has to collapse on the second page, it follows that $\pi_1^*$ is injective in cohomology, which implies that $\psi_3$ is nonzero and therefore $H^j(M)=0$ 
or $\Q$ for $j=3,4,5,6$. If these groups are $0$, then $M$ has the same cohomology ring as $S^2\times S^7$  {(hence it is rationally elliptic by Lemma \ref{L:RHT})}.

Assume finally that $H^j(M)=\Q$ for $j=3,4,5,6$ and let $z_i\in H^i(M)$ denote the generators of the nonzero degrees. In this case, we have that 
$\iota_1^*\oplus \iota_2^*:H^i(M)\to H^i(L_1)\oplus H^i(L_2)$ identifies $H^i(M)$ with $\ker \psi_i$ for $i=2,3,5$. Let $x_i\in H^2(L_i)$ and $y_i\in H^3(L_i)$ be generators, 
and notice that $x_iy_i:=x_i\cup y_i\in H^5(L_i)$ is also a generator. Let $(\iota_1^*\oplus \iota_2^*)(z_2)=(ax_1,bx_2)$ and $(\iota_1^*\oplus \iota_2^*)(z_3)=(cy_1,ey_2)$ 
be generators of $\ker \psi_2$ and $\ker \psi_3$, respectively. The injectivity of $\pi_1^*$ implies that $b,e\neq 0$. In fact, since $\pi_1^*(ax_1)=\Phi\circ \pi^*_2(bx_2)$, if we had $b=0$, then we would have $\pi_1^*(ax_1)=0$, from which $(ax_1, bx_2)=0$ contradicting the injectivity of $(\iota_1^*\oplus \iota_2^*):H^2(M)\to H^2(L_1)\oplus H^2(L_2)$. In a similar fashion we obtain $e\neq 0$. In particular, $(\iota_1\oplus\iota_2)^*(z_2z_3)=(acx_1y_1,be x_2y_2)$ is nonzero, hence $H^5(M)$ is generated by $z_2z_3$. 
By Poincar\'e Duality, the generators of $H^6(M)$, $H^7(M)$, and $H^9(M)$ are then $z_2z_4$, $z_3z_4$, and $z_2z_3z_4$, respectively. Thus the ring structure for $H^*(M)$ 
is either that of $H^*(\mathbb{C}\mathrm{P}^3\times S^3)$ (if $z_2^2\neq 0$, and hence we can take $z_4=z_2^2$) or $H^*(S^2\times S^3\times S^4)$ (if $z_2^2=0$). 
Taking $N\to M$, the $2$-connected circle bundle over $M$, we then get by the Gysin sequence that $H^*(N)$ is isomorphic to either $H^*(S^3\times S^7)$ or $H^*(S^3\times S^3\times S^4)$, 
respectively. Again,  {Lemma \ref{L:RHT} implies that $N$ is rationally elliptic, thus $M$ is rationally elliptic as well.}
\end{proof}

\subsection{Quotients with one non-orbifold point}\label{SS:one}

In this section, we assume that $M$ is a closed, simply connected, positively curved, cohomogeneity-three $G$-manifold with the following properties:
\begin{enumerate}
\item $X := M/G$ is homeomorphic to $S^3$.
\item There is exactly one non-orbifold point $p_1\in X$.
\item There is an edge of weight two.
\end{enumerate}

In this case, a very similar (but simpler) version of Lemma \ref{L:MV} can be obtained.

\begin{lemma}
Let $L_1$ denote the singular orbit corresponding to $p_1$, and let $E_1\to L_1$ denote the normal sphere bundle. Furthermore, let $L_2$ be a principal orbit, and $E_2\to L_2$ 
its normal sphere bundle. Then:
\begin{enumerate}
\item The inclusions $E_i\to M':=M\setminus (L_1\cup L_2)$ induce isomorphism in rational cohomology. In particular, there is an isomorphism of cohomology rings $\Phi:H^*(E_2)\to H^*(E_1)$.
\item There is a Mayer-Vietoris sequence in rational cohomology
\[
\ldots \to H^i(M)\stackrel{\iota_1^*\oplus \iota_2^*}{\to} H^i(L_1)\oplus H^i(L_2)\stackrel{\pi_1^*-\pi_2^*\circ \Phi}{\to} H^i(E_1)\to H^{i+1}(M)\to \ldots
\]
where $\pi_i:E_i\to L_i$ denotes the foot-point projection.
\end{enumerate}
\end{lemma}

A corollary similar to Corollary \ref{C:H1=0} also follows:

\begin{corollary}\label{C:H1=0-2}
Assume that $M$ is a $G$-manifold which satisfies properties (1)-(3) at the beginning of this section. Then the singular orbit $L_1$, the principal orbit $L_2$, and their sphere bundles $E_i$ 
have vanishing first rational cohomology. In addition, both $L_i$ and $E_i$ are orientable.
\end{corollary}

\begin{proof}
Let $n=\dim M$, so that $\dim E_i=n-1$, $\dim L_1=n-4$, and $\dim L_2=n-3$. From the Mayer-Vietoris sequence above, we have $H^{n-1}(E_i)\cong H^n(M)\cong \mathbb{Q}$, 
which implies that $E_i$ is orientable, and $H^1(E_i)\cong H^{n-2}(E_i)\cong H^{n-1}(M)\cong 0$ (here we used Poincar\'e Duality first, and Mayer-Vietoris second). 
This proves one of the statements. Again, from the Mayer-Vietoris sequence, we have that $0\to H^1(L_1)\oplus H^1(L_2)\to 0$ is exact, from which $H^1(L_i)=0$ follows.

It remains to prove that $L_i$ are orientable. The case of $L_2$ follows since $L_2$ is a principal orbit, hence diffeomorphic to a Lie group. The case of $L_1$ follows in the exact same way 
as in Corollary \ref{C:H1=0}.
\end{proof}

\begin{proposition}\label{P:one}
If $M$ is a $G$-manifold which satisfies Properties (1)-(3) at the beginning of Section \ref{SS:one}, then $M$ is rationally elliptic.
\end{proposition}

\begin{proof}

As in the proof of Proposition \ref{P:two}, we proceed by the dimension of $M$.

%
%

{\textbf{Case 1:}} The $6$-dimensional case. In this case, $L_2\simeq G$ has the rational cohomology of $S^3$, $L_1$ has the rational cohomology of $S^2$, and $E_1$ has the cohomology of $S^3\times S^2$. The Mayer-Vietoris sequence then reads
\[
0\to H^2(M)\to \Q\stackrel{\pi_1^*}{\to} \Q\to H^3(M)\to \Q\to\Q\to H^4(M)\to 0,\, H^5(M)\cong 0, H^6(M)\cong \Q.
\]
Since the spactral sequence for the fibration $S^3\to E_1\stackrel{\pi_1}{\to} L_1$ collapses at the second page, it follows that $\pi_1^*$ is injective, and in particular, an isomorphism in $H^2$. 
Hence $H^2(M)\cong H^4(M)\cong 0$, from which $H^3(M)\cong 0$ follows as well. Thus $H^*(M)$ is isomorphic to $H^*(S^6)$, hence $M$ is rationally elliptic by Corollary \ref{L:RHT}.

%
{\textbf{Case 2:}} The $7$-dimensional case. By Corollary \ref{C:H1=0-2}, this cannot happen since we would have $L_2\simeq G=S^3\times S^1$.

%
{\textbf{Case 3:}} The $9$-dimensional case. In this case, $G=L_2=S^3\times S^3$, and $L_1$ has the cohomology of $S^3\times S^2$. Moreover, $E_1$, being an $S^3$-bundle over $L_1$, 
must have the cohomology of $S^3\times S^3\times S^2$. The Mayer-Vietoris sequence then reads
\[
0\to H^2(M)\to \Q\stackrel{\pi_1^*}{\to} \Q\to H^3(M)\to \Q^3\stackrel{\pi_1^*-\pi_2^*\circ \Phi}{\to} \Q^2\to H^4(M)\to 0,
\]
\[
0\to H^5(M)\to \Q\stackrel{\pi_1^*}{\to} \Q^2\to H^6(M)\to \Q\stackrel{-\pi_2^*\circ \Phi}{\to} \Q\to H^7(M)\to 0.
\]
Since the spectral sequences for $\pi_1$ and $\pi_2$ degenerate at the second page, $\pi_1^*$ and $\pi_2^*$ are injective. In particular, we have $H^2(M)=H^4(M)=H^5(M)=H^7(M)=0$, 
and $H^6(M)\cong\Q$. By Poincar\'e Duality, $H^3(M)\cong\Q$ as well, and $M$ has the same rational cohomology ring as $S^3\times S^6$. 
By Lemma \ref{L:RHT}, it follows that $M$ is rationally elliptic.
\end{proof}

\section{Double disk bundle decomposition}\label{S:disk bundle}

Following \cite{GH87} we say that a space is \emph{rationally $\Omega$-elliptic} if the total rational homotopy group $\pi_*(\Omega X,\star)\otimes \Q$ is finite dimensional (here $\Omega X$ denotes the space of paths $\gamma:[0,1]\to X$ with $\gamma(0)=\gamma(1)=p_0$ fixed, and $\star$ denotes the constant path on $p_0$). Clearly, a space with trivial fundamental group is rationally elliptic if and only if it is rationally $\Omega$-elliptic. On the other hand, homogeneous spaces are rationally $\Omega$-elliptic even if they are not simply connected \cite[Page 430]{GH87}. Furthermore, if a fibration $E\to B$ has rationally $\Omega$-elliptic fibers, then $E$ is rationally $\Omega$-elliptic if and only if $B$ is.
\begin{proposition}\label{P:disk bundle}
Let $M$ be a $G$-manifold, and let $\pi:M\to M/G$ be the quotient map. Suppose the quotient space $X:=M/G$ has two subspaces $A_*,B_*$ such that $A:=\pi^{-1}(A_*)$
and $B:=\pi^{-1}(B_*)$ are smooth manifolds, and the complement $X\setminus B_{\epsilon}(A_*\cup B_*)$ is an orbifold diffeomorphic to a product $[0,1]\times \mathcal{O}$ for some compact 
$2$-orbifold $\mathcal{O}$. Then $M$ has a double disk bundle decomposition and hence it is rationally elliptic if either $A$ or $B$ is rationally $\Omega$-elliptic.
\end{proposition}

\begin{proof}
Fixing an orbifold diffeomorphism $\varphi:[0,1]\times \mathcal{O}\to X\setminus B_{\epsilon}(A_*\cup B_*)$, let $\partial \over \partial t$ denote the vector field in $[0,1]\times \mathcal{O}$ 
parallel to $[0,1]$, let $T=\varphi_*\left({\partial\over \partial t}\right)$ and finally let $\bar{T}$ denote the horizontal vector field on $M\setminus B_{\epsilon}(A\cup B)$ projecting to $T$. 
Then the flow $\Phi$ of $\bar{T}$ at time $1$ takes $\partial B_{\epsilon}(A)$ to $\partial B_{\epsilon}(B)$, and $M$ can be described as
\[
M=B_{\epsilon}(A)\cup \left(M\setminus B_{\epsilon}(A\cup B)\right)\cup B_{\epsilon}(B)\simeq B_{\epsilon}(A)\cup_{\Phi}B_{\epsilon}(B).
\]
By \cite[Corollary 6.1]{GH87}, it then follows that $M$ is rationally elliptic if and only if $\partial B_{\epsilon}(A)\simeq \partial B_{\epsilon}(B)$ is rationally $\Omega$-elliptic, which is equivalent to one of $A$, or $B$ being rationally $\Omega$-elliptic.
\end{proof}

\part{Applying the results of Part \ref{P:techniques} to the quotient space $M/G$}\label{P:applications}

The main goal of the second part of this paper is to prove Theorem \ref{main-thm:cohomogeneity three}. We begin with identifying possible configurations for the singular strata 
of the quotient space $M/G$. Next, for each configuration, we apply the appropriate technique from Part \ref{P:techniques} to prove that M is rationally elliptic.

\section{The structure of the singular strata of the quotient space}\label{S:singular strata}

In this section, we assume that $M$ is a closed, simply connected, positively curved Riemannian manifold and $G$ is a compact, connected Lie group acting on $M$ by isometries 
with cohomogeneity three and $\partial(M/G)=\emptyset$. Call $X=M/G$ the quotient Alexandrov space. This section is devoted to classifying possible singular strata configurations of $X$.

\begin{proposition}\label{P:1-dim stratum}
Let $X$ be as above, and assume it is not an orbifold. Then the possible singular strata correspond to the following labeled graphs (cf. Definition \ref{D:lab-graph}):

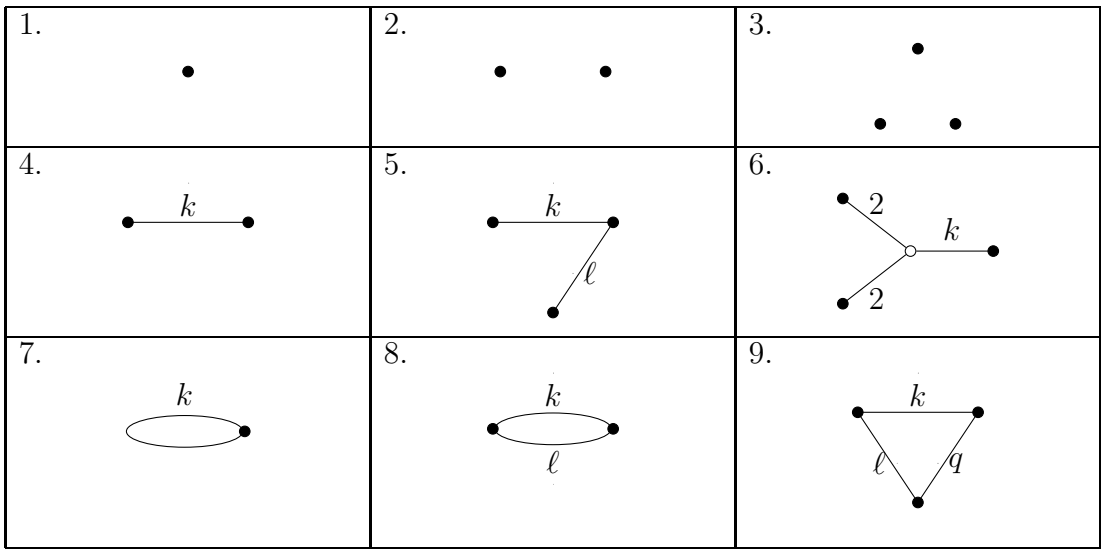
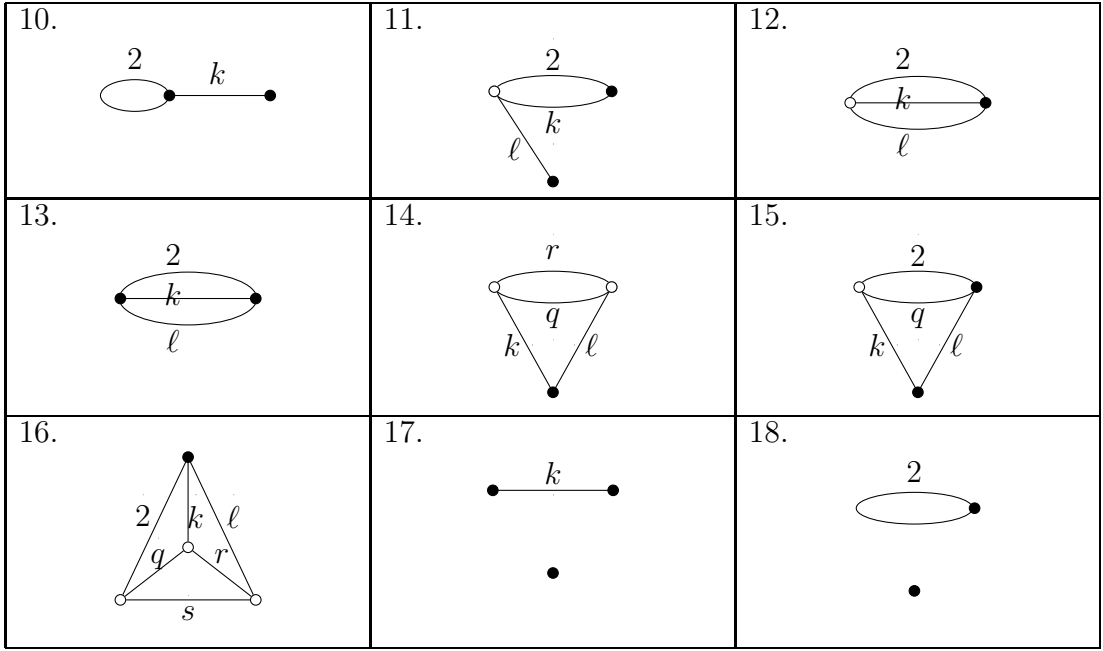
\begin{longtable}{ |p{4.5cm}|p{4.5cm}|p{4.5cm}|}
\hline
$1.$
\begin{center}
\begin{tikzpicture}
\filldraw[black] (0,0) circle (2pt);
\end{tikzpicture}
\end{center}
&
$2.$
\begin{center}
\begin{tikzpicture}
\filldraw[black] (-1.2,0) circle (2pt);
\filldraw[black] (0.2,0) circle (2pt);
\end{tikzpicture}
\end{center}
&
$3.$
\begin{center}
\begin{tikzpicture}
\filldraw[black] (-0.5,0) circle (2pt);
\filldraw[black] (-1,-1) circle (2pt);
\filldraw[black] (0,-1) circle (2pt);
\end{tikzpicture}
\end{center}\vspace{-0.5cm}\\ \hline

$4.$

\begin{center}
\begin{tikzpicture}
\draw (-0.8,1) -- (0.8,1);
\filldraw[black] (-0.8,1) circle (2pt);
\filldraw[black] (0.8,1) circle (2pt);
\filldraw[black] (0,1.52) circle (0pt) node[anchor=north] {$k$};
\end{tikzpicture}
\end{center}\vspace{-0.5cm}
&

$5.$
\begin{center}
\begin{tikzpicture}
\draw (-0.8,1) -- (0.8,1);
\draw (0,-0.2) -- (0.8,1);
\filldraw[black] (0,-0.2) circle (2pt);
\filldraw[black] (-0.8,1) circle (2pt);
\filldraw[black] (0.8,1) circle (2pt);
\filldraw[black] (0,1.52) circle (0pt) node[anchor=north] {$k$};
\filldraw[black] (0.26,0.32) circle (0pt) node[anchor=west] {$\ell$};
\end{tikzpicture}
\end{center}\vspace{-0.5cm}
&

$6.$

\begin{center}
\begin{tikzpicture}
\draw (-1.9,1) -- (-1,0.3) node[midway, above] {$2$};
\draw (-1,0.3) -- (0.1,0.3) node[midway, above] {$k$};
\draw  (-1.9,-0.4) -- (-1,0.3) node[midway, below] {$2$};
\filldraw[white] (-1,0.3) circle (2pt);
\draw[black] (-1,0.3) circle (2pt);
\filldraw[black] (-1.9,1) circle (2pt); 
\filldraw[black] (0.1,0.3) circle (2pt);
\filldraw[black] (-1.9,-0.4) circle (2pt);
\end{tikzpicture}
\end{center}\vspace{-0.5cm} \\ \hline

$7.$

\begin{center}
\begin{tikzpicture}
\draw (0,1) ellipse [x radius=22pt, y radius=6pt] node[yshift=0.2cm, above] {$k$};
\filldraw[black] (0.8,1) circle (2pt);
\end{tikzpicture}
\end{center}
&
$8.$

\begin{center}
\begin{tikzpicture}
\draw (0,1) ellipse [x radius=22pt, y radius=6pt];
\filldraw[black] (-0.8,1) circle (2pt);
\filldraw[black] (0.8,1) circle (2pt);
\filldraw[black] (0,1.74) circle (0pt) node[anchor=north] {$k$};
\filldraw[black] (0,0.26) circle (0pt) node[anchor=south] {$\ell$};
\end{tikzpicture}
\end{center}

&
$9.$

\begin{center}
\begin{tikzpicture}
\draw (-0.8,1) -- (0.8,1);
\draw (0,-0.2) -- (0.8,1);
\draw (0,-0.2) -- (-0.8,1);
\filldraw[black] (0,-0.2) circle (2pt);
\filldraw[black] (-0.8,1) circle (2pt);
\filldraw[black] (0.8,1) circle (2pt);
\filldraw[black] (0,1.52) circle (0pt) node[anchor=north] {$k$};
\filldraw[black] (0.26,0.32) circle (0pt) node[anchor=west] {$q$};
\filldraw[black] (-0.28,0.32) circle (0pt) node[anchor=east] {$\ell$};
\end{tikzpicture}
\end{center}\vspace{-0.2cm}\\ \hline

$10.$

\begin{center}
\begin{tikzpicture}
\draw (-0.6,1) -- (0.8,1) node[midway, above] {$k$};
\draw (-1,1) ellipse [x radius=13pt, y radius=6pt] node[yshift=0.2cm, above] {$2$};
\filldraw[black] (-0.54,1) circle (2pt);
\filldraw[black] (0.8,1) circle (2pt);
\end{tikzpicture}\vspace{-0.5cm}
\end{center}
&
$11.$

\begin{center}
\begin{tikzpicture}
\draw (0,-0.2) -- (-0.78,1);
\filldraw[black] (-0.28,0.23) circle (0pt) node[anchor=east] {$\ell$};
\filldraw[black] (0,1.7) circle (0pt) node[anchor=north] {$2$};
\filldraw[black] (0,0.3) circle (0pt) node[anchor=south] {$k$};
\filldraw[black] (0.78,1) circle (2pt);
\filldraw[white] (-0.78,1) circle (2pt);
\draw[black] (-0.78,1) circle (2pt);
\filldraw[black] (0,-0.2) circle (2pt);
\draw (0,1) ellipse [x radius=22pt, y radius=6pt];
\filldraw[white] (-0.78,1) circle (2pt);
\draw[black] (-0.78,1) circle (2pt);
\end{tikzpicture}\vspace{-0.5cm}
\end{center}

&
$12.$

\begin{center}
\begin{tikzpicture}
\hspace{-0.2cm}\draw (0,1) ellipse [x radius=26pt, y radius=10pt] node[yshift=0.3cm, above] {$2$} node[yshift=-0.27cm, below] {$\ell$};
\draw (-0.9,1) -- (0.9,1) node[midway, yshift=-0.23cm, above]{$k$};
\filldraw[white] (-.9,1) circle (2pt);
\draw[black] (-0.9,1) circle (2pt);
\filldraw[black] (0.9,1) circle (2pt);
\end{tikzpicture}\vspace{-0.5cm}
\end{center}
\vspace{-1cm}\\ \hline

$13.$

\begin{center}
\begin{tikzpicture}
\hspace{-0.2cm}\draw (0,1) ellipse [x radius=26pt, y radius=10pt] node[yshift=0.3cm, above] {$2$} node[yshift=-0.27cm, below] {$\ell$};
\draw (-0.9,1) -- (0.9,1) node[midway, yshift=-0.23cm, above]{$k$};
\filldraw[black] (-0.9,1) circle (2pt);
\filldraw[black] (0.9,1) circle (2pt);
\end{tikzpicture}\vspace{-0.5cm}
\end{center}\vspace{-0.5cm}
&
$14.$

\begin{center}
\begin{tikzpicture}
\draw (0,-0.4) -- (0.78,1);
\draw (0,-0.4) -- (-0.78,1);
\filldraw[black] (0.28,0.23) circle (0pt) node[anchor=west] {$\ell$};
\filldraw[black] (-0.28,0.23) circle (0pt) node[anchor=east] {$k$};
\filldraw[black] (0,1.7) circle (0pt) node[anchor=north] {$r$};
\filldraw[black] (0,0.3) circle (0pt) node[anchor=south] {$q$};
\filldraw[black] (0.78,1) circle (2pt);
\filldraw[black] (0,-0.4) circle (2pt);
\draw (0,1) ellipse [x radius=22pt, y radius=6pt];
\filldraw[white] (0.78,1) circle (2pt);
\draw[black] (0.78,1) circle (2pt);
\filldraw[white] (-0.78,1) circle (2pt) ;
\draw[black] (-0.78,1) circle (2pt) ;
\end{tikzpicture}
\end{center}\vspace{-0.5cm}

&
$15.$

\begin{center}
\begin{tikzpicture}
\draw (0,-0.4) -- (0.78,1);
\draw (0,-0.4) -- (-0.78,1);
\filldraw[black] (0.28,0.23) circle (0pt) node[anchor=west] {$\ell$};
\filldraw[black] (-0.28,0.23) circle (0pt) node[anchor=east] {$k$};
\filldraw[black] (0,1.7) circle (0pt) node[anchor=north] {$2$};
\filldraw[black] (0,0.3) circle (0pt) node[anchor=south] {$q$};
\filldraw[black] (0.78,1) circle (2pt);
\filldraw[black] (0,-0.4) circle (2pt);
\draw (0,1) ellipse [x radius=22pt, y radius=6pt];
\filldraw[white] (-0.78,1) circle (2pt) ;
\draw[black] (-0.78,1) circle (2pt) ;
\end{tikzpicture}
\end{center}\vspace{-0.5cm}

\\ \hline

$16.$

\begin{center}
\begin{tikzpicture}
\draw (-0.9,-0.8) -- (0.9,-0.8);
\draw (-0.9,-0.8) -- (0,1.1);
\draw (0.9,-0.8) -- (0,1.1);
\draw (-0.9,-0.8) -- (0,-0.1);
\draw (0.9,-0.8) -- (0,-0.1)  node[midway, above]{$r$};
\draw (0,-0.1) -- (0,1.1);
\filldraw[black] (0,1.1) circle (2pt);
\filldraw[white] (-0.9,-0.8) circle (2pt);
\draw[black] (-0.9,-0.8) circle (2pt);
\filldraw[white] (0.9,-0.8) circle (2pt);
\draw[black] (0.9,-0.8) circle (2pt);
\filldraw[white] (0,-0.1) circle (2pt);
\draw[black] (0,-0.1) circle (2pt);
\filldraw[black] (-0.6,0.6) circle (0pt) node[anchor=north] {$2$};
\filldraw[black] (0.1,0.6) circle (0pt) node[anchor=north] {$k$};
\filldraw[black] (0.6,0.6) circle (0pt) node[anchor=north] {$\ell$};
\filldraw[black] (0,-0.74) circle (0pt) node[anchor=north] {$s$};
\filldraw[black] (-0.4,0.02) circle (0pt) node[anchor=north] {$q$};
\end{tikzpicture}
\end{center}\vspace{-0.5cm} 

&
$17.$

\begin{center}
\begin{tikzpicture}
\draw (-0.8,0.7) -- (0.8,0.7);
\filldraw[black] (-0.8,0.7) circle (2pt);
\filldraw[black] (0.8,0.7) circle (2pt);
\filldraw[black] (0,-0.4) circle (2pt);
\filldraw[black] (0,0.64) circle (0pt) node[anchor=south] {$k$};
\end{tikzpicture}
\end{center}

&
$18.$

\begin{center}
\begin{tikzpicture}
\draw (0,1) ellipse [x radius=22pt, y radius=6pt] node[yshift=0.2cm, above] {$2$};
\filldraw[black] (0.8,1) circle (2pt);
\filldraw[black] (0,-0.1) circle (2pt);
\vspace{0.06cm}
\end{tikzpicture}
\end{center} \vspace{-0.5cm}\\ 
\hline
\end{longtable}
\end{proposition}

\begin{proof}
First suppose that the singular stratum $\Gamma$ of $X$ consists of isolated points: since these are small points, there must be at most three of them, and these are graphs 1, 2, 3.

Next, suppose that $\Gamma$ does not have any isolated points. Let $V$, $E$, and $C$ denote the number of vertices, edges, and independent cycles of $\Gamma$ (i.e. generators of $H_1(\Gamma,\Z)$), respectively.  The fact that for a simplicial complex $K$, $\sum_i(-1)^i\#\{i\textrm{-simplices}\}=\chi(K)=\sum_i(-1)^i\dim H_i(K)$ implies, in the case $K=\Gamma$, that:
\begin{equation}\label{eq:vecp}
V-E+C=\comp,
\end{equation} 
where $\comp=\dim H_0(\Sigma)$ denotes the number of connected components of $\Gamma$. Furthermore, the sum of the valencies of all the vertices equals twice the number of the edges 
(that is, $2E=V_1+2V_2+3V_3$ where $V_i$ is the number of vertices of valency $i$), thus equation \eqref{eq:vecp} becomes
\begin{equation}\label{eq:vcp}
\frac{V_1}{2}-\frac{V_3}{2}+C=\comp.
\end{equation}
Hence $V_1-V_3$ is even. Moreover, vertices of valency one or two are small points and thus $V_1+V_2\leq 3$. We break the proof into cases.

{\textbf{Case 1:}}  {$\Gamma$ does not have any cycles}. In this case, Equation \eqref{eq:vcp}, together with the facts that $V_1\leq 3$
and $b_0(\Sigma)\geq 1$, implies that 
$$2\leq V_3+2b_0(\Sigma)=V_1\leq 3.$$
Therefore, $b_0(\Sigma)=1$. Moreover, either $V_1=2$ and $V_3=0$, or $V_1=3$ and $V_3=1$. Since $V_1+V_2\leq 3$, by Equation \eqref{eq:vecp}, 
it follows that there are three possibilities:
\begin{itemize}
\item $V_1=2$, $V_2=0$, $V_3=0$, $\comp=1$, and $E=1$.  {This gives graph 4 in the list}.
\item $V_1=2$, $V_2=1$, $V_3=0$, $\comp=1$, and $E=2$.  {This gives graph 5}.
\item $V_1=3$, $V_2=0$, $V_3=1$, $\comp=1$, and $E=3$.  {In this case, the vertex of valency three cannot be a small point and thus the multiplicities of the edges coming out of it must be $(2,2,k)$. 
This gives graph 6.}
\end{itemize} 

{\textbf{Case 2:}}  {$\Gamma$ has} exactly one cycle. Since a vertex of valency one cannot lie on the singular cycle, 
after passing to the two-fold branched cover of $X$ along the singular cycle, we will have a positively curved Alexandrov space with $2V_1$ vertices of valency one,
and hence at least $2V_1+V_2$ small points. Therefore, $2V_1+V_2\leq 3$. Equation \eqref{eq:vcp} then implies that $2> V_1-V_3=2(\comp-1)$. Thus we have $\comp=1$ and $V_1=V_3$. 
Together with Equation \eqref{eq:vecp}, we get the following five possibilities:
\begin{itemize}
\item $V_1=V_3=0$ and $V_2=\comp=E=1$.  {This gives graph 7}.
\item $V_1=V_3=0$, $V_2=E=2$, and $\comp=1$.  {This gives graph 8}.
\item $V_1=V_3=0$, $V_2=E=3$, and $\comp=1$.  {This gives graph 9}.
\item $V_1=V_3=1$, $V_2=0$, $\comp=1$, and $E=2$.  {In this case, the (unlabeled) graph is the same as in graph 10, and the vertex of valency three must then have space of directions $S^2_{2,2,k}(\kappa)$ or $S^2_{2,3,3}(\kappa)$ ($\kappa=$1 or 4). The value $\kappa=1$ cannot occur for otherwise the cycle would be a quotient geodesic only crossing orbifold points, contradicting Lemma \ref{L:no-only-orb}. Furthemore, the space of directions $S^2_{2,3,3}(4)$ cannot occur either, for otherwise the triple branched cover of $X$ along the loop would have 4 small points, 
contradicting Hsiang-Kleiner. Thus only $S^2_{2,2,k}(4)$ is allowed, which gives graph 10.}
\item $V_1=V_3=1$, $V_2=1$, $\comp=1$, and $E=3$.  {In this case, the unlabaled graph is the same as graph 11, the vertex of valency three has space of directions $S^2_{p,k,\ell}(\kappa)$, 
and the loop consists of two edges of weights $p$ and $q$. First notice that $\kappa=1$ for otherwise the two-fold branched cover along the loop will have four small points. 
Similarly, if $\min(p,k)>2$ then the triple cover of $X$ along the loop would have four small points. Therefore, the only possibility is graph 11}. 
\end{itemize}

{\textbf{Case 3:}} $\Gamma$ has two independent cycles $c_1$ and $c_2$. In this case, after passing to the two-fold branched cover $Y:=X_2(c_1)$ 
and then either to the universal cover $\widetilde{Y}$ or to the two-fold branched cover of $Y$ along $b_X^{-1}(c_2\setminus c_1)$, 
we get a positively curved Alexandrov space with at least $4V_1$ vertices of valency one. Therefore, $V_1=0$ and hence $V_3=4-2\comp\leq 2$.
Moreover, given a vertex $v$ of valency two, there is a singular cycle that does not pass through $v$. After repeated application of passing to the two-fold branched covers
along such singular cycles or their preimages, we get a positively curved Alexandrov space with at least $2V_2$ small points and thus $V_2\leq 1$. 
Equation \eqref{eq:vecp} then yields the following possibilities:
\begin{itemize}
\item $V_1=0$, $V_2=0$, $V_3=2$, $\comp=1$, and $E=3$.  {The two vertices thus have space of directions $S^2_{k,l,m}(\kappa)$ with $\kappa=1$ or $4$. Since $\kappa$ cannot be 1  for both (otherwise they would both be orbifold points, contradicting Lemma \ref{L:no-only-orb}), one must have at least a vertex with $\kappa=4$. Depending on the value of $\kappa$ for the other vertex, 
we have graphs 12 and 13}.

\item $V_1=0$, $V_2=1$, $V_3=2$, $\comp=1$, and $E=4$.  {In this case, the (unlabeled) graph must be the same as graph 14. The two vertices thus have space of directions 
$S^2_{r,q,k}(\kappa_1)$ and $S^2_{r,q,\ell}(\kappa_2)$ respectively, with $\kappa_i=1$ or $4$, but one cannot have $\kappa_1=\kappa_2=4$ because  doubling $X$ along the length-two cycle would produce an Alexandrov space with four small points, contradicting Hsiang-Kleiner. Therefore, we can assume $\kappa_1=1$. If $\kappa_2=1$ as well, we get labeled graph 14. 
Finally, if $\kappa_2=4$, at least one of the weights in the length-two cycle $c$ have to be two for otherwise the triple branched cover $X_3(c)$ would have once again four small points. 
Therefore, the only labeled graph in this case is graph 15.}
\end{itemize}

{\textbf{Case 4:}} $\Gamma$ has $C\geq 3$ independent cycles. As in Case 3, we have $V_1=0$.
Suppose there exists a vertex $v$ of valency two. Then there are $C-1\geq 2$ independent cycles that do not pass through $v$. Choose two of them, and call them $c_1, c_2$. 
Taking the iterated two-fold branched cover along $c_1$ and $c_2$, we get an Alexandrov space with at least four small points given by the preimages of $v$, contradicting Hsiang-Kleiner. 
Hence, $V_2=0$. Since by assumption, $X$ is not an orbifold, we conclude that $X$ contains at least one non-orbifold point of valency three.

For any vertex $w$ of valency three, we claim that we can find at least $C-2$ independent cycles that do not pass through $w$. In fact, choose a set $\mathcal{B}$ of $C$ independent cycles 
with the minimum number of elements passing through $w$, and assume by contradiction that there are $c_1, c_2, c_3\in \mathcal B$ all passing though $w$. 
Then, at least one of $c_i\pm c_j$ will not pass through $w$ anymore, and upon substituting $c_i$ with $c_i\pm c_j$ we will obtain a new independent set with strictly less cycles through $w$, 
a contradiction. Fixing a non-orbifold point $w$ of valency three and taking iterated double covers over the $(C-2)$ cycles as before, the preimage of $w$ then consists of $2^{C-2}$ 
small points hence $2^{C-2}<3$, from which $C=3$ and there is at least one cycle $c$ not passing through $w$. Assume, by contradiction, that  there is a second non-orbifold point $w'$ 
with valency three. In this case, the double branched cover $X_2(c)$ would be a positively curved Alexandrov space with at least three small points (the two preimages of $w$, 
and at least one preimage of $w'$) whose singular graph $\Gamma_2(c)$ has at least two independent cycles $\hat{c}_1,\, \hat{c}_2$. We claim that at least one cycle 
$\hat{c}$ of $\Gamma_2(c)$ does not pass through all small points (call them $\hat{w}_1,\hat{w}_2, \hat{w}_3$): if one of $\hat{c}_1,\, \hat{c}_2$ does not pass through all three points 
we are done, so assume that both $\hat{c}_1$ and $\hat{c}_2$ pass through all three of them, and orient $\hat{c}_1, \, \hat{c}_2$ in such a way that they meet the three points in increasing order. 
For $i=1,2$, let $\hat{d}_i$  be the segment of $\hat{c}_i$ from $\hat{w}_1$ to $\hat{w}_2$, not passing through $\hat{w}_3$. 
 Without loss of generality, we may assume that $\hat{d}_1$ and $\hat{d}_2$ do not coincide. Then the concatenation $\hat{c}:=\hat{d}_1*\hat{d}_2^{-1}$ 
is a cycle in $\Gamma_2(c)$ not passing through $\hat{w}_3$, as claimed. Taking the double branched cover over $\hat{c}$ would result in a positively curved 
Alexandrov space with at least four small points, a contradiction, and thus there must be a unique non-orbifold point with valency three.

By Equation \eqref{eq:vecp} and \eqref{eq:vcp}, a graph with $C=3$ (hence $V_1=V_2=0$) and $V_3\geq 1$ must have:
\begin{itemize}
\item $V_1=V_2=0$, $V_3=4$, $\comp=1$, and $E=6$.
\end{itemize}
This leads to the following unlabeled graphs:
\\

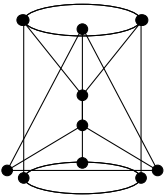
\begin{tikzpicture}
\hspace{1.8cm}\draw (-1,-0.2) -- (1,-0.2);
\draw (-1,-0.2) -- (0,1.7);
\draw (1,-0.2) -- (0,1.7);
\draw (-1,-0.2) -- (0,0.4);
\draw (1,-0.2) -- (0,0.4);
\draw (0,0.4) -- (0,1.7);
\filldraw[black] (0,1.68) circle (2pt) node[anchor=south] {};
\filldraw[black] (-1,-0.2) circle (2pt) node[anchor=east] {};
\filldraw[black] (1,-0.2) circle (2pt) node[anchor=west] {};
\filldraw[black] (0,0.4) circle (2pt) node[anchor=north] {};
\hspace{4.84cm}\draw (0,0.8) -- (0.8,1.8);
\draw (0,0.8) -- (-0.8,1.8);
\filldraw[black] (0.8,1.8) circle (2pt) node[anchor=west] {};
\filldraw[black] (-0.8,1.8) circle (2pt) node[anchor=east] {};
\filldraw[black] (0,0.8) circle (2pt) node[anchor=north] {};
\filldraw[black] (0,-0.1) circle (2pt) node[anchor=north] {};
\draw (0,1.8) ellipse [x radius=22pt, y radius=6pt];
\draw (0,0.8) -- (0,-0.1);
\draw (0,-0.3) ellipse [x radius=22pt, y radius=6pt];
\hspace{4.84cm}\draw (0,1.8) ellipse [x radius=22pt, y radius=6pt];
\draw (0,-0.3) ellipse [x radius=22pt, y radius=6pt];
\draw (0.78,-0.3) -- (0.78,1.8);
\draw (-0.78,-0.3) -- (-0.78,1.8);
\filldraw[black] (0.78,-0.3) circle (2pt) node[anchor=west] {};
\filldraw[black] (0.78,1.8) circle (2pt) node[anchor=east] {};
\filldraw[black] (-0.78,.-0.3) circle (2pt) node[anchor=north] {};
\filldraw[black] (-0.78,1.8) circle (2pt) node[anchor=north] {};
\end{tikzpicture}

\vspace{0.28cm}

\noindent The second and the third configurations cannot occur since after passing to the two-fold branched cover of $X$ along any singular cycle 
that does not contain the non-orbifold point, we get a positively curved Alexandrov space with two non-orbifold points of valency three and at least three independent cycles,
which contradicts the discussion in the previous paragraph. As for the first configuration, by the discussion above there must be at most one vertex of curvature 4. Thus the only possibility is graph 16.

\vspace{0.2cm}

It remains to discuss the case in which $\Gamma$ consists of both $1$-dimensional singular strata and isolated points. Let $V_0$ denote the number of isolated points. 
Note that, if $V_3^s$ denotes the number of non-orbifold points of valency three, then $V_0+V_1+V_2+V_3^s\leq 3$. Note, moreover, that if $X$ has more than one singular cycle, then $V_0=0$ by the two-fold branched cover arguments. 
Therefore, the number of singular cycles of $X$ is at most one. 

 If $X$ has no cycles, then $V_1\geq 2$ by Equation \eqref{eq:vcp} but this is possible only if $V_0=1$, $V_1=2$, and $V_2=V_3^s=0$. Therefore, the only possibility is graph 17.

Furthermore, if $X$ has exactly one singular cycle, then every vertex of valency two lies on the cycle and hence by passing to the two-fold branched cover of $X$ along the singular cycle, 
we get $2V_0+2V_1+V_2\leq 3$. This can occur only if $V_0=1$, $V_1=0$, and $V_2\leq 1$, but the case $V_2=0$ cannot occur since otherwise the cycle would be consisting of orbifold points, contradicting Lemma \ref{L:no-only-orb}. The only possibility is then graph 18, where the only closed edge $c$ must have multiplicity $k=2$ since the $k$-fold cover of $X_k(c)$ has $k+1$ small points.
\end{proof}

\section{Proof of Theorem \ref{main-thm:cohomogeneity three}}\label{S:proof}

In this section, we prove Theorem \ref{main-thm:cohomogeneity three}, dividing it into cases as in Proposition \ref{P:1-dim stratum}.
\begin{proposition}\label{P:main-1}
Let $M$ be a closed, simply connected, positively curved Riemannian manifold which admits a cohomogeneity-three action by a compact, connected Lie group $G$. If the structure of the singular strata
of the quotient space $M/G$ is given by one among graphs $10, 11, 12, 13, 14, 15, 16$, or $18$ from Proposition \ref{P:1-dim stratum}, then $M$ is rationally elliptic.
\end{proposition}

 {\begin{proof}
All the graphs listed have an edge with weight two, and one or two non-orbifold points. Therefore, the result follows from Propositions \ref{P:one} and \ref{P:two}.
\end{proof}}

The remaining cases will be shown to be rationally elliptic, by showing that they can be written as double disk bundles and then applying Proposition \ref{P:disk bundle}. 
Before doing that however, we need the following lemma to help us recognize orbifolds diffeomorphic to a product $S^2_{p_0,p_1}\times [0,1]$, for $S^2_{p_0,p_1}$ a spindle orbifold.

\begin{lemma}\label{L:hom-to-diff}
Let $S^2_{p_0,p_1}$ be the spindle orbifold, and let $\mathcal{O}$ denote a Riemannian orbifold with boundary admitting a \emph{weight-preserving homeomorphism} $f: \mathcal{O}\to S^2_{p_0,p_1}\times[0,1]$, that is, a homeomorphism sending points in $\mathcal{O}$ to points in $S^2_{p_0,p_1}\times [0,1]$ with the same isotropy group. Then $\mathcal{O}$ is orbifold-diffeomorphic to $S^2_{p_0,p_1}\times [0,1]$.
\end{lemma}

\begin{proof}
Fix Riemannian metrics on $\mathcal{O}$ and $S^2_{p_0,p_1}\times [0,1]$, and fix orientations so that $f$ is orientation-preserving. Let $\Sigma_i$ (resp. $\Sigma_i'$), $i=0,1$, 
denote the orbifold stratum of $\mathcal O$ (resp. of $S^2_{p_0,p_1}\times[0,1]$) with local group $\mathbb{Z}_{p_{i}}$. We then have:
\begin{enumerate}
\item The restriction $f|_{\Sigma_1^c}:\Sigma_1^c\to (\Sigma_1')^c$ lifts to a $\mathbb{Z}_{p_{0}}$-equivariant homeomorphism between the orbifold covers
\[
\tilde{f}_0:\widetilde{\Sigma^c_1}\to \widetilde{(\Sigma_1')^c}\stackrel{\textrm{diff}}{\simeq}\mathbb{D}^2\times [0,1].
\]
Since these are 3-dimensional manifolds, it follows that there is an orientation-preserving $\mathbb{Z}_{p_{0}}$-equivariant diffeomorphism $\tilde{F}_0: \widetilde{B_\epsilon(\Sigma_0)}\to \widetilde{B_\epsilon(\Sigma'_0)}$. One way to produce $\tilde{F}_0$ is to first deform the homeomorphism $\tilde f_0$ to an orientation-preserving diffeomorphism $F'_0:\widetilde{\Sigma^c_1}\to \widetilde{(\Sigma_1')^c}$ isotopic to $\tilde{f}_0$ sending $\Sigma_{0}$ to $\Sigma'_{0}$ and inducing an isometry $dF_0':\nu \Sigma_{0}\to \nu\Sigma'_{0}$. Then $dF'_0$ is automatically $\mathbb{Z}_{p_{0}}$-equivariant, 
and we can define, for $v\in \nu^1\Sigma_0$, $\tilde{F}_0(\exp tv):=\exp(t dF'_0(v))$.

In particular, this gives an orbifold diffeomorphism ${F}_0:{B_\epsilon(\Sigma_0)}\to {B_\epsilon(\Sigma_0')}$. Similarly, one can obtain an orbifold diffeomorphism $F_1:{B_\epsilon(\Sigma_1)}\to {B_\epsilon(\Sigma_1')}$.
\item  The restriction $f|_{(\Sigma_0\cup \Sigma_1)^c}:(\Sigma_0\cup \Sigma_1)^c\to(\Sigma_0'\cup \Sigma_1')^c$ is a homeomorphism between smooth manifolds, 
thus it can be isotoped to a (orientation-preserving) diffeomorphism $F:(\Sigma_0\cup \Sigma_1)^c\to(\Sigma_0'\cup \Sigma_1')^c$. 
Furthermore, we can assume that $F$ takes $B_\epsilon(\Sigma_i)\setminus \Sigma_i$ to $B_\epsilon(\Sigma_i')\setminus \Sigma_i'$.
\end{enumerate}
Finally, we claim that $F_i|_{\partial B_\epsilon(\Sigma_i)}$ and $F|_{\partial B_\epsilon(\Sigma_i)}$ are isotopic diffeomorphisms in the group $\operatorname{Diff}^+(\partial B_\epsilon(\Sigma_i), \partial B_\epsilon(\Sigma_i'))$ of orientation-preserving diffeomorphisms, so that it is possible to change $F_i$ in a neighbourhood of $\partial B_\epsilon(\Sigma_i)$ 
in such a way that it becomes equal to $F$ along $\partial B_\epsilon(\Sigma_i)$, and the maps $F_0, F, F_1$ glue together to a diffeomorphism 
$$\bar{F}:\mathcal{O}\to S^2_{p_0,p_1}\times[0,1].$$

In order to prove the claim, we will in fact prove that $\pi_0(\operatorname{Diff}^+(\partial B_\epsilon(\Sigma_i), \partial B_\epsilon(\Sigma_i')))=\{0\}$. 
Notice first that $\partial B_\epsilon(\Sigma_i)$ and $\partial B_\epsilon(\Sigma_i')$ are both manifolds diffeomorphic to $S^1\times [0,1]$, hence they admit orientation-preserving embeddings 
$\eta: \partial B_\epsilon(\Sigma_i)\to S^2$ and $\eta':\partial B_\epsilon(\Sigma_i')\to S^2$ with image the complement of two $\epsilon$-disks $D_\pm$.
The maps $\eta'\circ F_i\circ \eta^{-1}|_{\partial D^{\pm}}:\partial D^{\pm}\to \partial D^{\pm}$ are orientation-preserving diffeomorphisms of $S^1$ and therefore they extend to diffeomorphisms $D^\pm\to D^\pm$. 
In particular, the restriction map
\[
\operatorname{Diff}^+_D(S^2,S^2)\to \operatorname{Diff}^+(\partial B_\epsilon(\Sigma_i), \partial B_\epsilon(\Sigma_i')),
\]
where $\operatorname{Diff}^+_D(S^2,S^2)$ is the group of orientation-preserving diffeomorphisms of $S^2$ that sends $D^\pm$ to itself
, induces a surjection $\pi_0(\operatorname{Diff}^+_D(S^2,S^2))\to\pi_0(\operatorname{Diff}^+(\partial B_\epsilon(\Sigma_i), \partial B_\epsilon(\Sigma_i')))$.
 On the other hand, since any two points in $S^2$ can be moved to any other two points via a diffeomorphism isotopic to the identity, we have
 \[
 \pi_0(\operatorname{Diff}^+_D(S^2,S^2))\cong \pi_0(\operatorname{Diff}^+(S^2,S^2))\cong \pi_0(SO(3))\cong \{0\}
\]
and the final claim is proved. 
%
%
\end{proof}

\begin{proposition}\label{P:main-2}
Suppose $M$ is a closed, simply connected, positively curved Riemannian manifold which admits a cohomogeneity-three action by a compact, connected Lie group $G$. If the structure of the singular strata 
of the quotient space $M/G$ is given by one among graphs $1, 2, 3, 4, 5, 6, 8, 9$, or $17$ from Proposition \ref{P:1-dim stratum}, then $M$ is rationally elliptic.
\end{proposition}
 
\begin{proof}
We will prove that a manifold $M$ such that $X=M/G$ has singular structure corresponding to a graph in the list above, can be written as a double disk bundle, and the rational ellipticity can be deduced from Proposition \ref{P:disk bundle}. 
In order to do this, we have to find subspaces $A_*, B_*$ in $X$ whose preimages are smooth submanifolds of $M$, such that $X':=X\setminus B_{\epsilon}(A_*\cup B_*)$ is an orbifold 
splitting off a one-dimensional factor, see Figure \ref{F:DDB}.

{\bf Graphs 1 and 2.} Take $A^*$ to be one singular point, and $B_*$ to be either a principal point (Graph 1) or the other singular point (Graph 2). 
Then $A:=\pi^{-1}(A_*)$ and $B:=\pi^{-1}(B_*)$ are orbits (hence rationally $\Omega$-elliptic smooth manifolds), and $X'$ is a manifold with boundary, diffeomorphic to $S^2\times [0,1]$.

{\bf Graph 3.} Take $A_*$ to be one singular point, and $B_*$ to be a geodesic $\gamma$ between the other two singular points. Then $A$ is a smooth manifold because it is an orbit. 
We claim that $B$ is smooth as well. In fact:
\begin{itemize}
\item The preimage of $\gamma|_{[0,\epsilon)}$ and $\gamma|_{(1-\epsilon,1]}$ is a smooth manifold near the endpoints since it is given as the exponential image of a (rank two) sub-bundle 
of $\nu(\pi^{-1}(\gamma(0)))$ and $\nu(\pi^{-1}(\gamma(1)))$.
\item The preimage of $\gamma|_{(0,1)}$ is smooth as well, because $\gamma(0,1)$ is contained in the manifold part of $X$, 
over which $\pi$ is a smooth Riemannian submersion. 
\end{itemize}
Finally $X'$ is a manifold with boundary, diffeomorphic to $S^3\setminus\left({B_\epsilon(p_1)\coprod B_\epsilon(p_2)}\right)\simeq S^2\times [0,1]$.

{\bf Graph 4.} Take $A_*, B_*$ to be the two vertices of $\Gamma$. Their preimages are orbits and in particular they are both smooth and rationally $\Omega$-elliptic. 
Furthermore, $X'$ is homeomorphic to $S^2\times [0,1]$, and the orbifold points consist of one line, joining the two boundary components. 
Therefore, $X'$ is weight-preserving homeomorphic to $S^2_{k,1}\times [0,1]$, hence by Lemma \ref{L:hom-to-diff} it is orbifold-diffeomorphic to it.

{\bf Graph 5.} Consider a curve $\gamma:[0,1]\to X$ with the following properties:
\begin{itemize}
\item $\gamma$ starts at the point with space of directions $S^2_{1,k}$ with $\gamma^+(0)$ being the furthest point (i.e. at distance $\pi/2$) from the orbifold point of $S^2_{1,k}$.
\item $\gamma$ ends at the point with space of directions $S^2_{1,\ell}$ with $\gamma^-(1)$ being the furthest point (i.e. at distance $\pi/2$) from the orbifold point of $S^2_{1,\ell}$.
\item $\gamma|_{(0,1)}$ is contained in the manifold part of $X$, it is smooth on $(0,1)$, and it is a geodesic in an $\epsilon$ neighborhood of $0$ and $1$.
\item Letting $\gamma_k$ and $\gamma_\ell$ denote the singular edges of multiplicity $k$ and $\ell$, then the concatenation $\gamma*\gamma_k*\gamma_\ell$ is unknotted in $X\simeq S^3$. This can be done for example by defining $\gamma$ as a concatenation of curves ``running parallel'' to $\gamma_k$ and $\gamma_\ell$.

\end{itemize}
Define now $A_*\subset X$ as the singular vertex not contained in $\gamma$ and $B_*:=\gamma([0,1])$. Then $A=\pi^{-1}(A_*)$ is smooth and rationally $\Omega$-elliptic because it is an orbit, and $B=\pi^{-1}(B_*)$ is smooth for the same reason as the submanifold $B$ in the case of Graph 3.
we can find a homeomorphism $X\to S^3$ sending $B_\epsilon(A_*)$ to a neighbourhood of the north pole, $B_{\epsilon}(B_*)$ to a neighborhood of the south pole, 
and $\gamma_k,\gamma_\ell$ to two meridians. In particular, $X'$ is again homeomorphic to $S^2\times [0,1]$, with two 1-dimensional orbifold strata consisting of lines 
$\{x\}\times [0,1]$, $\{y\}\times [0,1]$ hence, by Lemma \ref{L:hom-to-diff}, $X'$ is orbifold diffeomorphic to $S^2_{k,\ell}\times [0,1]$.

{\bf Graph 6.} This is similar to graph 5: take $A_*$ the edge of weight $k$, and $B_*$ the segment between the two vertices with spaces of directions $S^2_{1,2}$ such that:
\begin{itemize}
\item $\gamma$ starts at the point with space of directions $S^2_{1,2}$ with $\gamma^+(0)$ being the furthest point (i.e. at distance $\pi/2$) from the orbifold point of $S^2_{1,2}$.
\item $\gamma$ ends at the point with space of directions $S^2_{1,2}$ with $\gamma^-(1)$ being the furthest point (i.e. at distance $\pi/2$) from the orbifold point of $S^2_{1,2}$.
\item $\gamma|_{(0,1)}$ is contained in the manifold part of $X$, it is smooth on $(0,1)$, and it is a geodesic in an $\epsilon$ neighbourhood of $0$ and $1$.
\item Letting $\gamma_1$ and $\gamma_2$ denote the singular edges of multiplicity two, then the concatenation $\gamma*\gamma_1*\gamma_2$ is unknotted in $X\simeq S^3$.
\end{itemize}
The sets $A=\pi^{-1}(A_*)$ and $B=\pi^{-1}(B_*)$ are smooth manifolds for the same reasons as before. Furthermore, the $G$ action restricts to a cohomogeneity-one action on them, and thus they are rationally $\Omega$-elliptic by \cite{GH87}. By the unknottedness assumption, 
we can find a homeomorphism $X\to S^3$ sending $B_\epsilon(A_*)$ to a neighborhood of the north pole, $B_{\epsilon}(B_*)$ to a neighborhood of the south pole, 
and $\gamma_1,\gamma_2$ to two meridians. In particular, $X'$ is homeomorphic to $S^2\times [0,1]$, with two 1-dimensional orbifold strata consisting of lines $\{x\}\times [0,1]$, $\{y\}\times [0,1]$. Again by Lemma \ref{L:hom-to-diff}, $X'$ is orbifold diffeomorphic to $S^2_{2,2}\times [0,1]$.

{\bf Graph 8.} In this case, there is a loop $c$, which we claim is the unknot. In fact, if $c$ was knotted, then $X_2(c)$ would have non-trivial fundamental group and two small points.
Thus its universal cover would be a positively curved Alexandrov space with $2|\pi_1(X_2(c))|\geq 4$ small points, a contradiction. Therefore, $c$ is the unknot. 
Then by letting $A_*, B_*$ denote the two vertices, we can find a homeomorphism $X\to S^3$ sending $B_\epsilon(A_*)$ to a neighborhood of the north pole, $B_{\epsilon}(B_*)$ 
to a neighborhood of the south pole, and the two edges to meridians. By Lemma \ref{L:hom-to-diff}, $X'$ is orbifold diffeomorphic to $S^2_{k,l}\times [0,1]$.

{\bf Graph 9.} Letting $B_*$ denote an edge of $\Gamma$ and $A_*$ the opposite vertex, the argument is completely analogous to the case of graph 5, once we show that the loop $c$ 
in graph 9 is the unknot. In fact, the universal cover $\widetilde{X_2(c)}$ of $X_2(c)$ is an Alexandrov space with positive curvature and $3\cdot |\pi_1(X_2(c))|$ small points, 
from which we get that $\pi_1(X_2(c))$ is trivial, thus $c$ is the unknot.

{\bf Graph 17.} In this case, we can take $B_*$ to be the edge, and $A_*$ the isolated point. The same argument as for graph 5 shows that $A$ is a smooth submanifold. 
Furthermore, there are no other singularities outside of $A_*\cup B_*$ hence $X'$ is diffeomorphic to $S^2\times[0,1]$.
\begin{figure}
\begin{longtable}{ |p{4.5cm}|p{4.5cm}|p{4.5cm}|}
\hline
Graph 1.
\begin{center}
\begin{tikzpicture}
\filldraw[black] (-1,0) circle (2pt) node[anchor=east] {$A_*$};
\draw (1,0)  circle (2pt) node[anchor=west] {$B_*$};
\end{tikzpicture}
\end{center}
\vspace{-0.5cm}
&
Graph 2.
\begin{center}
\begin{tikzpicture}
\filldraw[black] (-1,0) circle (2pt) node[anchor=east] {$A_*$};
\filldraw[black] (1,0) circle (2pt) node[anchor=west] {$B_*$};
\end{tikzpicture}
\end{center}
\vspace{-0.5cm}
&
Graph 3.
\begin{center}
\begin{tikzpicture}
\filldraw[black] (-1,0) circle (2pt) node[anchor=east] {$A_*$};
\filldraw[black] (1,0.5) circle (2pt);
\filldraw[black] (1,-0.5) circle (2pt);
\draw[dashed] (1,-0.5) -- node[anchor=west] {$B_*$}++ (0,1);
\end{tikzpicture}
\end{center}
\vspace{-0.5cm}
\\\hline

Graph 4.
\begin{center}
\begin{tikzpicture}
\draw (-1,0) -- (1,0);
\filldraw[black] (-1,0) circle (2pt) node[anchor=east] {$A_*$};
\filldraw[black] (1,0) circle (2pt) node[anchor=west] {$B_*$};
\end{tikzpicture}
\end{center}
\vspace{-0.5cm}
&

Graph 5.
\begin{center}
\begin{tikzpicture}
\draw (-1,0) -- (1,0.5);
\draw (-1,0) -- (1,-0.5);
\draw[dashed] (1,0.5) -- node[right] {$B_*$}++ (0,-1) ;
\filldraw[black] (1,0.5) circle (2pt);
\filldraw[black] (1,-0.5) circle (2pt);
\filldraw[black] (-1,0) circle (2pt) node[anchor=east] {$A_*$};
\end{tikzpicture}
\end{center}
\vspace{-0.5cm}
&
Graph 6.
\vspace{-0.2cm}
\begin{center}
\begin{tikzpicture}
\draw (0,0) -- (1,0.5) node[midway, above] {2};;
\draw (-1,0) -- (0,0) node[midway, above] {$A_*$};
\draw  (0,0) -- (1,-0.5) node[midway, below] {2};
\draw[dashed] (1,0.5) -- node[right] {$B_*$}++(0,-1);
\filldraw[white] (-1,0) circle (2pt);
\draw[black] (-1,0) circle (2pt)  node[anchor=east] {\phantom{$A_*$}};
\filldraw[black] (0,0) circle (2pt); 
\filldraw[black] (1,-0.5) circle (2pt);
\filldraw[black] (1,0.5) circle (2pt);
\end{tikzpicture}
\end{center}
\vspace{-0.7cm}
\\\hline
Graph 8.

\begin{center}
\begin{tikzpicture}
\draw (0,0) ellipse [x radius=1cm, y radius=0.3cm];
\filldraw[black] (-1,0) circle (2pt) node[anchor=east] {$A_*$};
\filldraw[black] (1,0) circle (2pt) node[anchor=west] {$B_*$};
\end{tikzpicture}
\end{center}
\vspace{-0.5cm}
&
Graph 9.
\begin{center}
\begin{tikzpicture}
\draw (-1,0) -- (1,0.5);
\draw (-1,0) -- (1,-0.5);
\draw (1,0.5) -- node[right] {$B_*$}++ (0,-1) ;
\filldraw[black] (1,0.5) circle (2pt);
\filldraw[black] (1,-0.5) circle (2pt);
\filldraw[black] (-1,0) circle (2pt) node[anchor=east] {$A_*$};
\end{tikzpicture}
\end{center}
\vspace{-0.5cm}
&
Graph 17.
\begin{center}
\begin{tikzpicture}
\filldraw[black] (-1,0) circle (2pt) node[anchor=east] {$A_*$};
\filldraw[black] (1,0.5) circle (2pt);
\filldraw[black] (1,-0.5) circle (2pt);
\draw (1,-0.5) -- node[anchor=west] {$B_*$}++ (0,1);
\end{tikzpicture}
\end{center}
\vspace{-0.5cm}
\\ \hline
\end{longtable}
\caption{Double disk bundle decompositions. Dashed lines (resp. the hollow point in the first picture) represent curves (resp. point) chosen in the regular part.}\label{F:DDB}
\end{figure}
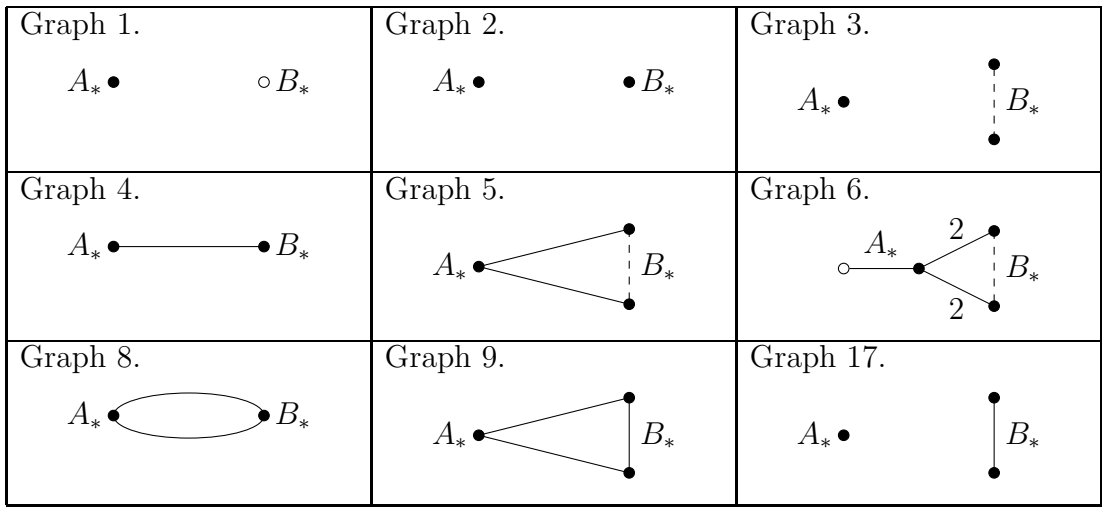
\end{proof}

The last case to study is the case of a quotient space whose singular stratum is as in graph 7, in the list of Proposition \ref{P:1-dim stratum}.

\begin{lemma}\label{L:knot}
Suppose $X=M/G$ is the quotient of a colosd, simply connected, positively curved, cohomogeneity-three manifold $M$, with $\partial X=\emptyset$ and labeled singular graph $\Gamma$ given by
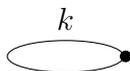
\begin{center}
\begin{tikzpicture}
\draw (0,1) ellipse [x radius=22pt, y radius=6pt] node[yshift=0.2cm, above] {$k$};
\filldraw[black] (0.8,1) circle (2pt);
\end{tikzpicture}
\end{center}
Then either the singular cycle $c$ in $\Gamma$ is the unknot, or it is the trefoil knot and $k=2$.
\end{lemma}

\begin{proof}
The universal cover of the two-fold branched cover $X_2(c)$ has $|\pi_1(X_2(c))|$ small points, therefore $|\pi_1(X_2(c))|\leq 3$. Furthermore, $X_2(c)$ is a $\mathbb{Z}_2$-homology sphere 
by \cite[Lemma 5.3]{GW14}. Hence either $X_2(c)$ is simply connected (thus $c$ is the unknot) or it is the unique lens space with fundamental group $\mathbb{Z}_3$. 
This uniquely determines a knot (\cite[Lemma 5.4]{GW14}) which is a \emph{2-bridge knot} and in fact the trefoil knot \cite[Section 10.D]{Rol03}.

If $c$ is a trefoil knot, suppose that the weight of the edge is $\geq 3$. Then the triple branched cover $X_3(c)$ would still be a positively curved Alexandrov space with one small point, 
and $\pi_1(X_3(c))=Q_8$ the quaternion $8$-group (cf. \cite[Section 10.D]{Rol03}). In particular, the universal cover of $X_3(c)$ would be positively curved Alexandrov space with $|Q_8|=8$ small points, 
a contradiction.
\end{proof}

\begin{proposition}
Let $M$, $X=M/G$, be as in Lemma \ref{L:knot}. Then $M$ is rationally elliptic.
\end{proposition}
\begin{proof}
From Lemma \ref{L:knot}, the singular loop $c$ is either the unknot, or it is the trefoil knot and $k=2$. The knotted case follows from Proposition \ref{P:one}, since we have an edge with weight two. Assume now that $c$ is the unknot, and take $A_*$ to be the non-orbifold point and $B_*$ to be a point in the singular edge. Clearly, the preimages of $A_*$ and $B_*$ in $M$ are orbits 
hence smooth. Furthermore, since $c$ is the unknot, we can find a homeomorphism $X\to S^3$ mapping $A_*$ to the north pole, $B_*$ to the south pole, and the rest of $c$ onto two meridians. 
In particular, $X'$ admits a strata and weight-preserving homeomorphism to $S^2_{k,k}\times [0,1]$ hence by Lemma \ref{L:hom-to-diff} it is diffeomorphic to it, and the result follows from Proposition \ref{P:disk bundle}.
\end{proof}

\appendix
\section{List of actions providing the 3-dimensional quotients}


When proving the main theorem, we showed that for all of the possible candidates of weighted graphs in the quotients, the corresponding manifold is rationally elliptic, without checking which of those candidates can indeed occur. For the convenience of an interested reader we provide, in Table \ref{Tb:Examples-list}, examples of cohomogeneity-three $G$-actions on positively curved manifolds exhibiting most of the candidates of weighted graphs shown in Proposition \ref{P:1-dim stratum}. We do not know if every example from Proposition \ref{P:1-dim stratum} can be achieved, or whether it can be done so with $G$ a connected group.  Whenever possible, we provide examples where $G$ is connected. The possible weighted graphs for which we do not know any action, are provided in Table \ref{Tb:unknown-examples}. In the tables, $\pi:SU(2)\to SO(3)$ denotes the standard double cover, $G_{pqr}\subset SO(3)$ are the subgroups in Table \ref{Ta:groups}, and $G_{pqr}^*:=\pi^{-1}(G_{pqr})\subset SU(2)$.

In all cases, the quotient space is homeomorphic to $S^3$. Furthermore, the singular set is \emph{planar}, i.e. contained in some $S^2$ embedded in $S^3$, 
with the exception of the second action with singular structure given by
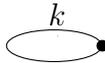
\[
\begin{tikzpicture}
\draw (0,1) ellipse [x radius=18pt, y radius=6pt];
\filldraw[black] (0.64,1) circle (2pt);
\filldraw[black] (0,0.55) circle (0pt);
\filldraw[black] (0.04,1.13) circle (0pt) node[anchor=south] {$k$};
\end{tikzpicture}
\]
where the singular loop forms a trefoil knot in $M/G$.

\newpage

\begin{center}

\renewcommand{\arraystretch}{1.5}
\begin{longtable}{|c|c|c|c|}
\hline
\multicolumn{1}{|c|}{Configuration} & \multicolumn{1}{|c|}{$M$} & \multicolumn{1}{|c|}{$G$} & \multicolumn{1}{|c|}{Description of the action}\\
\hline
\multicolumn{1}{|c|}{\multirow{2}{*}{\begin{tikzpicture}
\filldraw[black] (0,0) circle (2pt);
\filldraw[black] (0,-0.55) circle (0pt);
\end{tikzpicture}}} & \multicolumn{1}{|c|}{$S^6\subset \R^3\oplus \HH$} & \multicolumn{1}{|c|}{${\mathrm{SU}}(2)$} & \multicolumn{1}{|l|}{\hspace{-0.08cm}$q\cdot(v,w)=(qv\bar{q},qw)$}\\
\hline
\multicolumn{1}{|c|}{\multirow{2}{*}{\begin{tikzpicture}
\filldraw[black] (-1,0) circle (2pt);
\filldraw[black] (0,0) circle (2pt);
\filldraw[black] (0,-0.55) circle (0pt);
\end{tikzpicture}}} & \multicolumn{1}{|c|}{$S^4\subset \mathbb{C}^2\oplus \R$} & \multicolumn{1}{|c|}{$S^1$} & \multicolumn{1}{|l|}{\hspace{-0.08cm}$z\cdot(w_1,w_2,t)=(zw_1,zw_2,t)$}\\
\hline
\multicolumn{1}{|c|}{\multirow{2}{*}{\begin{tikzpicture}
\filldraw[black] (-0.5,-0.3) circle (2pt);
\filldraw[black] (-1,-1) circle (2pt);
\filldraw[black] (0,-1) circle (2pt);
\end{tikzpicture}}} & \multicolumn{1}{|c|}{\multirow{2}{*}{$S^5\subset \C^3$}} & \multicolumn{1}{|c|}{\multirow{2}{*}{$T^2$}} & \multicolumn{1}{|l|}{\multirow{2}{*}{\hspace{-0.11cm}$(z_1,z_2)\cdot(w_1,w_2,w_3)=(z_1w_1,z_2w_2,z_1z_2w_3)$}}\\
\multicolumn{1}{|c|}{} & \multicolumn{1}{|c|}{} & \multicolumn{1}{|c|}{} & \multicolumn{1}{|c|}{}\\
\hline
\multicolumn{1}{|c|}{\begin{tikzpicture}
\draw (-1,1) -- (0,1);
\filldraw[black] (-1,1) circle (2pt);
\filldraw[black] (0,1) circle (2pt);
\filldraw[black] (0,1.1) circle (0pt);
\filldraw[black] (-0.5,1.52) circle (0pt) node[anchor=north] {$k$};
\end{tikzpicture}} & \multicolumn{1}{|c|}{$S^4\subset \mathbb{C}^2\oplus \R$} & \multicolumn{1}{|c|}{$S^1$} & \multicolumn{1}{|l|}{\hspace{-0.08cm}$z\cdot(w_1,w_2,t)=(zw_1,z^kw_2,t)$}\\
\hline
\multicolumn{1}{|c|}{\multirow{2}{*}{\begin{tikzpicture}
\draw (-1,1) -- (0,1);
\draw (0,1) -- (-0.5,0.1);
\filldraw[black] (-0.5,0.1) circle (2pt);
\filldraw[black] (-1,1) circle (2pt);
\filldraw[black] (0,1) circle (2pt);
\filldraw[black] (-0.5,1.52) circle (0pt) node[anchor=north] {$k$};
\filldraw[black] (-0.36,0.44) circle (0pt) node[anchor=west] {$\ell$};
\end{tikzpicture}}} & \multicolumn{1}{|c|}{\multirow{2}{*}{$S^5\subset \C^3$}} & \multicolumn{1}{|c|}{\multirow{2}{*}{$T^2$}} & \multicolumn{1}{|l|}{\multirow{2}{*}{\hspace{-0.11cm}$(z_1,z_2)\cdot(w_1,w_2,w_3)=(z_2w_1,z_1^kz_2^{\ell}w_2,z_1w_3)$}}\\
\multicolumn{1}{|c|}{} & \multicolumn{1}{|c|}{} & \multicolumn{1}{|c|}{} & \multicolumn{1}{|c|}{}\vspace{-0.5cm}\\
\multicolumn{1}{|c|}{} & \multicolumn{1}{|c|}{} & \multicolumn{1}{|c|}{} & \multicolumn{1}{|c|}{}\\
\hline
\multicolumn{1}{|c|}{} & \multicolumn{1}{|c|}{} & \multicolumn{1}{|c|}{} & \multicolumn{1}{|c|}{}\vspace{-0.5cm}\\
\multicolumn{1}{|c|}{\multirow{2}{*}{\begin{tikzpicture}
\draw (-1.8,0.9) -- (-1,0.3) node[midway, above] {$2$};
\draw (-1,0.3) -- (0,0.3) node[midway, above] {$2$};
\draw  (-1.8,-0.3) -- (-1,0.3) node[midway, below] {$2$};
\filldraw[white] (-1,0.3) circle (2pt);
\draw[black] (-1,0.3) circle (2pt);
\filldraw[black] (-1.8,0.9) circle (2pt); 
\filldraw[black] (0,0.3) circle (2pt);
\filldraw[black] (-1.8,-0.3) circle (2pt);
\end{tikzpicture}}} & \multicolumn{1}{|c|}{\multirow{2}{*}{${\mathrm{SU}(3)}/{T^2}$}} & \multicolumn{1}{|c|}{\multirow{2}{*}{$\mathrm{SO}(3)$}} & \multicolumn{1}{|l|}{\multirow{2}{*}{\hspace{-0.08cm}$A\cdot[B]=[AB]$}}\\
\multicolumn{1}{|c|}{} & \multicolumn{1}{|c|}{} & \multicolumn{1}{|c|}{} & \multicolumn{1}{|c|}{}\vspace{-0.5cm}\\
\multicolumn{1}{|c|}{} & \multicolumn{1}{|c|}{} & \multicolumn{1}{|c|}{} & \multicolumn{1}{|c|}{}\\
\hline
\multicolumn{1}{|c|}{\multirow{4}{*}{
\begin{tikzpicture}
\draw (0,1) ellipse [x radius=18pt, y radius=6pt];
\filldraw[black] (0.64,1) circle (2pt);
\filldraw[black] (0,0.55) circle (0pt);
\filldraw[black] (0.04,1.13) circle (0pt) node[anchor=south] {$k$};
\end{tikzpicture}}} & \multicolumn{1}{|c|}{$S^6\subset \R^3\oplus \HH$} & \multicolumn{1}{|c|}{$\mathrm{SU}(2)\times\Z_k$} & \multicolumn{1}{|l|}{\hspace{-0.08cm}$(q,g)\cdot(v,w)=(qv\bar{q},qw\bar g)$}\vspace{-0.08cm}\\
\multicolumn{1}{|c|}{} & \multicolumn{1}{|c|}{} & \multicolumn{1}{|c|}{} & \multicolumn{1}{|c|}{}\vspace{-0.6cm}\\\cline{2-4}
\multicolumn{1}{|c|}{} & \multicolumn{1}{|c|}{\multirow{3}{*}{$S^5\subset \C^3$}} & \multicolumn{1}{|c|}{\multirow{3}{*}{$(T^2\rtimes\Z_3)\rtimes\Z_2$}} & \multicolumn{1}{|l|}{\hspace{-0.11cm}$(z_1,z_2)\cdot(w_1,w_2,w_3)=(z_1w_1,z_2w_2,\bar{z}_1\bar{z}_2w_3)$}\vspace{-0.08cm}\\
\multicolumn{1}{|c|}{} & \multicolumn{1}{|c|}{} & \multicolumn{1}{|c|}{} & \multicolumn{1}{|l|}{\hspace{-0.08cm}$g\cdot(w_1,w_2,w_3)=(gw_2,gw_3,gw_1)$,~$\Z_3=\langle g\rangle$}\vspace{-0.08cm}\\
\multicolumn{1}{|c|}{} & \multicolumn{1}{|c|}{} & \multicolumn{1}{|c|}{} & \multicolumn{1}{|l|}{\hspace{-0.08cm}$h\cdot(w_1,w_2,w_3)=(\bar{w}_1,\bar{w}_3,\bar{w}_2)$,~$\Z_2=\langle h\rangle$}\\
\hline
\multicolumn{1}{|c|}{\multirow{2}{*}{\begin{tikzpicture}
\draw (0,1) ellipse [x radius=18pt, y radius=6pt];
\filldraw[black] (-0.64,1) circle (2pt);
\filldraw[black] (0.64,1) circle (2pt);
\filldraw[black] (0,1.74) circle (0pt) node[anchor=north] {$k$};
\filldraw[black] (0,0.26) circle (0pt) node[anchor=south] {$\ell$};
\filldraw[black] (0.1,-0.36) circle (0pt) node[anchor=south] {$(k,\ell)=d$};
\end{tikzpicture}}} & \multicolumn{1}{|c|}{\multirow{2}{*}{$S^4\subset \mathbb{C}^2\oplus \R$}} & \multicolumn{1}{|c|}{\multirow{2}{*}{$S^1\times \Z_d$}} & \multicolumn{1}{|l|}{\multirow{2}{*}{\hspace{-0.27cm}$\begin{array}{l}(z,g)\cdot(w_1,w_2,t)=(g^az^{k/d}w_1,g^bz^{{\ell}/d}w_2,t)\vspace{-0.12cm}\\ \hspace{0.05cm}a\ell-bk=d\end{array}$}}\\
\multicolumn{1}{|c|}{} & \multicolumn{1}{|c|}{} & \multicolumn{1}{|c|}{} & \multicolumn{1}{|c|}{}\\
\multicolumn{1}{|c|}{} & \multicolumn{1}{|c|}{} & \multicolumn{1}{|c|}{} & \multicolumn{1}{|c|}{}\\
\hline
\multicolumn{1}{|c|}{\multirow{3}{*}{\begin{tikzpicture}
\draw (-1,1) -- (0,1);
\draw (0,1) -- (-0.5,0.1);
\draw (-1,1) -- (-0.5,0.1);
\filldraw[black] (-0.5,0.1) circle (2pt);
\filldraw[black] (-1,1) circle (2pt);
\filldraw[black] (0,1) circle (2pt);
\filldraw[black] (-0.5,1.52) circle (0pt) node[anchor=north] {$k$};
\filldraw[black] (-0.36,0.44) circle (0pt) node[anchor=west] {$m$};
\filldraw[black] (-0.64,0.44) circle (0pt) node[anchor=east] {$\ell$};
\filldraw[black] (-0.4,-0.7) circle (0pt) node[anchor=south] {$(k,\ell,m)=1$};
\end{tikzpicture}}} & \multicolumn{1}{|c|}{\multirow{3}{*}{$S^5\subset \C^3$}} & \multicolumn{1}{|c|}{\multirow{3}{*}{$T^2$}} & \multicolumn{1}{|l|}{\multirow{2}{*}{\hspace{-0.27cm}$\begin{array}{l}(z_1,z_2)\cdot(w_1,w_2,w_3)=(z_2^{d}w_1,z_1^{a}z_2^{c}w_2,z_1^{b}z_2^{e}w_3)\vspace{-0.12cm}\\ \hspace{0.05cm}d=\gcd(\ell, m)=\mu m+\lambda \ell,\vspace{-0.12cm}\\ \hspace{0.05cm}a=m/d,\, b=-\ell/d,\, c=\lambda k,\, e=\mu k  \end{array}$}}\\
\multicolumn{1}{|c|}{} & \multicolumn{1}{|c|}{} & \multicolumn{1}{|c|}{} & \multicolumn{1}{|c|}{}\\
\multicolumn{1}{|c|}{} & \multicolumn{1}{|c|}{} & \multicolumn{1}{|c|}{} & \multicolumn{1}{|c|}{}\\
\hline
\multicolumn{1}{|c|}{\multirow{2}{*}{\begin{tikzpicture}
\draw (-0.6,1) -- (0.6,1);
\draw (-1,1) ellipse [x radius=13pt, y radius=6pt];
\filldraw[black] (-0.54,1) circle (2pt);
\filldraw[black] (0.6,1) circle (2pt);
\filldraw[black] (0,0.5) circle (0pt);
\filldraw[black] (0.04,0.9) circle (0pt) node[anchor=south] {$k$};
\filldraw[black] (-0.94,1.13) circle (0pt) node[anchor=south] {$2$};
\end{tikzpicture}}} & \multicolumn{1}{|c|}{\multirow{2}{*}{$S^4\subset \mathbb{C}^2\oplus \R$}} & \multicolumn{1}{|c|}{\multirow{2}{*}{$S^1\rtimes\Z_2$}} & \multicolumn{1}{|l|}{\hspace{-0.08cm}$z\cdot(w_1,w_2,t)=(zw_1,z^kw_2,t)$}\vspace{-0.08cm}\\
\multicolumn{1}{|c|}{} & \multicolumn{1}{|c|}{} & \multicolumn{1}{|c|}{} & \multicolumn{1}{|l|}{\hspace{-0.08cm}$g\cdot(w_1,w_2,t)=(\bar{w}_1,\bar{w}_2,-t)$,~$\Z_2=\langle g\rangle$}\\
\hline
\multicolumn{1}{|c|}{\multirow{2}{*}{\begin{tikzpicture}
\draw (0,1) ellipse [x radius=18pt, y radius=6pt];
\draw (-0.64,1) -- (0,-0.1);
\filldraw[black] (-0.73,0.4) circle (0pt) node[anchor=west] {$2$};
\filldraw[black] (0,1.7) circle (0pt) node[anchor=north] {$3$};
\filldraw[black] (0,0.3) circle (0pt) node[anchor=south] {$2$};
\filldraw[black] (0.64,1) circle (2pt);
\filldraw[white] (-0.64,1) circle (2pt);
\draw[black] (-0.64,1) circle (2pt);
\filldraw[black] (0,-0.1) circle (2pt);
\end{tikzpicture}}} & \multicolumn{1}{|c|}{\multirow{3}{*}{$\C\mathrm{P}^3$}} & \multicolumn{1}{|c|}{\multirow{3}{*}{$\mathrm{SO}(3)$}} & \multicolumn{1}{|l|}{\multirow{3}{*}{\hspace{-0.26cm}$\begin{array}{l}[A]\cdot[p(z,w)]=[A\cdot p(z,w)]\\p(z,w)\in\C[z,w]_3\end{array}$}}\vspace{-0.08cm}\\
\multicolumn{1}{|c|}{} & \multicolumn{1}{|c|}{} & \multicolumn{1}{|c|}{} & \multicolumn{1}{|l|}{}\\
\multicolumn{1}{|c|}{} & \multicolumn{1}{|c|}{} & \multicolumn{1}{|c|}{} & \multicolumn{1}{|c|}{}\\
\hline
\multicolumn{1}{|c|}{} & \multicolumn{1}{|c|}{} & \multicolumn{1}{|c|}{} & \multicolumn{1}{|l|}{}\vspace{-0.5cm}\\
\multicolumn{1}{|c|}{\begin{tikzpicture}
\draw (0,1) ellipse [x radius=23pt, y radius=10pt];
\draw (-0.8,1) -- (0.8,1);
\filldraw[black] (0.8,1) circle (2pt);
\filldraw[white] (-0.8,1) circle (2pt);
\draw[black] (-0.8,1) circle (2pt);
\filldraw[black] (0,1.8) circle (0pt) node[anchor=north] {$p$};
\filldraw[black] (0,1.3) circle (0pt) node[anchor=north] {$q$};
\filldraw[black] (0,0.24) circle (0pt) node[anchor=south] {$r$};
\end{tikzpicture}} & \multicolumn{1}{|c|}{} & \multicolumn{1}{|c|}{} & \multicolumn{1}{|l|}{}\vspace{-1.5cm}\\
\multicolumn{1}{|c|}{} & \multicolumn{1}{|c|}{$S^6\subset \im\HH\oplus \HH$} & \multicolumn{1}{|c|}{$SU(2)\times G_{pqr}^*$} & \multicolumn{1}{|l|}{\hspace{-0.11cm}$(A,g)\cdot(x,y)=(\pi(A)x,Ayg^{-1})$}\\
\multicolumn{1}{|c|}{} & \multicolumn{1}{|c|}{} & \multicolumn{1}{|c|}{} & \multicolumn{1}{|c|}{}\\
\hline
\newpage
\hline
\multicolumn{1}{|c|}{} & \multicolumn{1}{|c|}{} & \multicolumn{1}{|c|}{} & \multicolumn{1}{|l|}{}\vspace{-0.7cm}\\
\multicolumn{1}{|c|}{\begin{tikzpicture}
\draw (0,1) ellipse [x radius=23pt, y radius=10pt];
\draw (-0.8,1) -- (0.8,1);
\filldraw[black] (0.8,1) circle (2pt);
\filldraw[black] (-0.8,1) circle (2pt);
\filldraw[black] (0,1.8) circle (0pt) node[anchor=north] {$p$};
\filldraw[black] (0,1.3) circle (0pt) node[anchor=north] {$q$};
\filldraw[black] (0,0.24) circle (0pt) node[anchor=south] {$r$};
\end{tikzpicture}} &  \multicolumn{1}{|c|}{} & \multicolumn{1}{|c|}{} & \multicolumn{1}{|l|}{}\vspace{-1.5cm}\\
\multicolumn{1}{|c|}{} & \multicolumn{1}{|c|}{$S^4\subset \mathbb{H}\oplus \R$} & \multicolumn{1}{|c|}{$S^1\cdot G_{pqr}^*$} & \multicolumn{1}{|l|}{\hspace{-0.11cm}$(ge^{i\theta})\cdot(x,t)=(ge^{i\theta}x,t)$}\\
\multicolumn{1}{|c|}{} & \multicolumn{1}{|c|}{} & \multicolumn{1}{|c|}{} & \multicolumn{1}{|l|}{}\vspace{-0.2cm}\\
\hline
\multicolumn{1}{|c|}{\multirow{2}{*}{\begin{tikzpicture}
\draw (0,1) ellipse [x radius=18pt, y radius=6pt];
\draw (-0.64,1) -- (0,-0.1);
\draw (0.64,1) -- (0,-0.1);
\filldraw[black] (-0.18,0.3) circle (0pt) node[anchor=east] {$2$};
\filldraw[black] (0.19,0.3) circle (0pt) node[anchor=west] {$2$};
\filldraw[black] (0,1.7) circle (0pt) node[anchor=north] {$2$};
\filldraw[black] (0,0.3) circle (0pt) node[anchor=south] {$2$};
\filldraw[black] (0.64,1) circle (2pt);
\filldraw[white] (-0.64,1) circle (2pt);
\draw[black] (-0.64,1) circle (2pt);
\filldraw[black] (0,-0.1) circle (2pt);
\end{tikzpicture}}} & \multicolumn{1}{|c|}{\multirow{3}{*}{$S^5\subset \C^3$}} & \multicolumn{1}{|c|}{\multirow{3}{*}{$\begin{array}{c}(T^2\times  G)\rtimes H\\ G\cong H\cong \Z_2\end{array}$}} & \multicolumn{1}{|l|}{\hspace{-0.11cm}$(z_1,z_2)\cdot(w_1,w_2,w_3)=(z_1w_1,z_2w_2,z_1z_2w_3)$}\vspace{-0.08cm}\\
\multicolumn{1}{|c|}{} & \multicolumn{1}{|c|}{} & \multicolumn{1}{|c|}{} & \multicolumn{1}{|l|}{\hspace{-0.08cm}$g\cdot(w_1,w_2,w_3)=(w_1,-w_2,w_3)$}\vspace{-0.08cm}\\
\multicolumn{1}{|c|}{} & \multicolumn{1}{|c|}{} & \multicolumn{1}{|c|}{} & \multicolumn{1}{|l|}{\hspace{-0.08cm}$h\cdot(w_1,w_2,w_3)=(\bar{w}_2,\bar{w}_1,\bar{w}_3)$}\\
\hline
%
\multicolumn{1}{|c|}{\multirow{2}{*}{\begin{tikzpicture}
\draw (0,1) ellipse [x radius=18pt, y radius=6pt];
\draw (-0.64,1) -- (0,-0.1);
\draw (0.64,1) -- (0,-0.1);
\filldraw[black] (-0.18,0.3) circle (0pt) node[anchor=east] {$k d$};
\filldraw[black] (0.19,0.3) circle (0pt) node[anchor=west] {$\ell d$};
\filldraw[black] (0,1.7) circle (0pt) node[anchor=north] {$2$};
\filldraw[black] (0,0.3) circle (0pt) node[anchor=south] {$2$};
\filldraw[black] (0.64,1) circle (2pt);
\filldraw[white] (-0.64,1) circle (2pt);
\draw[black] (-0.64,1) circle (2pt);
\filldraw[black] (0,-0.1) circle (2pt);
\filldraw[black] (-0.08,-1) circle (0pt) node[anchor=south] {$(k,\ell)=1$};
\end{tikzpicture}}} & \multicolumn{1}{|c|}{\multirow{3}{*}{$S^5\subset \C^3$}} & \multicolumn{1}{|c|}{\multirow{3}{*}{$\begin{array}{c}(T^2\times G)\rtimes H \\ G\cong \Z_{d},\, H\cong \Z_2\end{array}$}} & \multicolumn{1}{|l|}{\hspace{-0.11cm}$(z_1,z_2)\cdot(w_1,w_2,w_3)=(z_2^{\ell}w_1,z_1z_2^{k}w_2,z_1w_3)$}\\
\multicolumn{1}{|c|}{} & \multicolumn{1}{|c|}{} & \multicolumn{1}{|c|}{} & \multicolumn{1}{|l|}{\hspace{-0.08cm}$g\cdot(w_1,w_2,w_3)=(gw_1,gw_2,w_3)$}\vspace{-0.08cm}\\
\multicolumn{1}{|c|}{} & \multicolumn{1}{|c|}{} & \multicolumn{1}{|c|}{} & \multicolumn{1}{|l|}{\hspace{-0.08cm}$h\cdot(w_1,w_2,w_3)=(\bar{w}_1,w_3,w_2)$}\\
\multicolumn{1}{|c|}{} & \multicolumn{1}{|c|}{} & \multicolumn{1}{|c|}{} & \multicolumn{1}{|c|}{}\\
\hline
\multicolumn{1}{|c|}{\multirow{2}{*}{\begin{tikzpicture}
\draw (0,1) ellipse [x radius=18pt, y radius=6pt];
\draw (-0.64,1) -- (0,-0.1);
\draw (0.64,1) -- (0,-0.1);
\filldraw[black] (-0.18,0.3) circle (0pt) node[anchor=east] {$2$};
\filldraw[black] (0.19,0.3) circle (0pt) node[anchor=west] {$2$};
\filldraw[black] (0,1.7) circle (0pt) node[anchor=north] {$2$};
\filldraw[black] (0,0.3) circle (0pt) node[anchor=south] {$2$};
\filldraw[white] (0.64,1) circle (2pt);
\draw[black] (0.64,1) circle (2pt);
\filldraw[white] (-0.64,1) circle (2pt);
\draw[black] (-0.64,1) circle (2pt);
\filldraw[black] (0,-0.1) circle (2pt);
\end{tikzpicture}}} & \multicolumn{1}{|c|}{\multirow{3}{*}{$S^4\subset \mathbb{C}^2\oplus \R$}} & \multicolumn{1}{|c|}{\multirow{3}{*}{$\begin{array}{c}(S^1\rtimes G)\times H\\ G\cong H\cong \Z_2\end{array}$}} & \multicolumn{1}{|l|}{\hspace{-0.08cm}$z\cdot(w_1,w_2,t)=(zw_1,zw_2,t)$}\vspace{-0.08cm}\\
\multicolumn{1}{|c|}{} & \multicolumn{1}{|c|}{} & \multicolumn{1}{|c|}{} & \multicolumn{1}{|l|}{\hspace{-0.08cm}$g\cdot(w_1,w_2,t)=(\bar{w}_2,\bar{w}_1,-t)$}\vspace{-0.08cm}\\
\multicolumn{1}{|c|}{} & \multicolumn{1}{|c|}{} & \multicolumn{1}{|c|}{} & \multicolumn{1}{|l|}{\hspace{-0.08cm}$h\cdot(w_1,w_2,t)=(-w_2,-w_1,t)$}\\
\hline
\multicolumn{1}{|c|}{\multirow{2}{*}{\begin{tikzpicture}
\draw (0,1) ellipse [x radius=18pt, y radius=6pt];
\draw (-0.64,1) -- (0,-0.1);
\draw (0.64,1) -- (0,-0.1);
\filldraw[black] (-0.18,0.3) circle (0pt) node[anchor=east] {$k d$};
\filldraw[black] (0.19,0.3) circle (0pt) node[anchor=west] {$\ell d$};
\filldraw[black] (0,1.7) circle (0pt) node[anchor=north] {$2$};
\filldraw[black] (0,0.3) circle (0pt) node[anchor=south] {$2$};
\filldraw[white] (0.64,1) circle (2pt);
\draw[black] (0.64,1) circle (2pt);
\filldraw[white] (-0.64,1) circle (2pt);
\draw[black] (-0.64,1) circle (2pt);
\filldraw[black] (0,-0.1) circle (2pt);
\filldraw[black] (-0.08,-0.9) circle (0pt) node[anchor=south] {$(k,\ell)=1$};
\end{tikzpicture}}} & \multicolumn{1}{|c|}{\multirow{3}{*}{$S^4\subset \mathbb{C}^2\oplus \R$}} & \multicolumn{1}{|c|}{\multirow{3}{*}{$\begin{array}{c}(S^1\times G)\rtimes H \\ G\cong \Z_{d},\, H\cong \Z_2\end{array}$}} & \multicolumn{1}{|l|}{\hspace{-0.08cm}$z\cdot(w_1,w_2,t)=(z^{k}w_1,z^{\ell}w_2,t)$}\vspace{-0.08cm}\\
\multicolumn{1}{|c|}{} & \multicolumn{1}{|c|}{} & \multicolumn{1}{|c|}{} & \multicolumn{1}{|l|}{\hspace{-0.08cm}$g\cdot(w_1,w_2,t)=(gw_1,w_2,t)$}\vspace{-0.08cm}\\
\multicolumn{1}{|c|}{} & \multicolumn{1}{|c|}{} & \multicolumn{1}{|c|}{} & \multicolumn{1}{|l|}{\hspace{-0.08cm}$h\cdot(w_1,w_2,t)=(\bar{w}_1,\bar{w}_2,-t)$}\vspace{-0.2cm}\\
\multicolumn{1}{|c|}{} & \multicolumn{1}{|c|}{} & \multicolumn{1}{|c|}{} & \multicolumn{1}{|c|}{}\\
\hline
\multicolumn{1}{|c|}{\multirow{2}{*}{\begin{tikzpicture}
\draw (-1,1) -- (0,1);
\filldraw[black] (-1,1) circle (2pt);
\filldraw[black] (0,1) circle (2pt);
\filldraw[black] (-0.5,0.2) circle (2pt);
\filldraw[black] (-0.5,1.52) circle (0pt) node[anchor=north] {$k$};
\end{tikzpicture}}} & \multicolumn{1}{|c|}{\multirow{2}{*}{$S^5\subset \C^3$}} & \multicolumn{1}{|c|}{\multirow{2}{*}{$T^2$}} & \multicolumn{1}{|l|}{\multirow{2}{*}{\hspace{-0.11cm}$(z_1,z_2)\cdot(w_1,w_2,w_3)=(z_2w_1,z_1^kz_2w_2,z_1w_3)$}}\\
\multicolumn{1}{|c|}{} & \multicolumn{1}{|c|}{} & \multicolumn{1}{|c|}{} & \multicolumn{1}{|c|}{}\vspace{-0.6cm}\\
\multicolumn{1}{|c|}{} & \multicolumn{1}{|c|}{} & \multicolumn{1}{|c|}{} & \multicolumn{1}{|c|}{}\\
\hline
\multicolumn{1}{|c|}{\multirow{2}{*}{\begin{tikzpicture}
\draw (0,1) ellipse [x radius=18pt, y radius=6pt];
\filldraw[black] (0.64,1) circle (2pt);
\filldraw[black] (0,0.1) circle (2pt);
\filldraw[black] (0.04,1.13) circle (0pt) node[anchor=south] {$2$};
\end{tikzpicture}}} & \multicolumn{1}{|c|}{\multirow{2}{*}{$S^5\subset \C^3$}} & \multicolumn{1}{|c|}{\multirow{2}{*}{$T^2\rtimes\Z_2$}} & \multicolumn{1}{|l|}{\hspace{-0.11cm}$(z_1,z_2)\cdot(w_1,w_2,w_3)=(z_1w_1,z_2w_2,z_1z_2w_3)$}\vspace{-0.08cm}\\
\multicolumn{1}{|c|}{} & \multicolumn{1}{|c|}{} & \multicolumn{1}{|c|}{} & \multicolumn{1}{|l|}{\hspace{-0.08cm}$g\cdot(w_1,w_2,w_3)=(\bar{w}_2,\bar{w}_1,\bar{w}_3)$,~$\Z_2=\langle g\rangle$}\vspace{-0.3cm}\\
\multicolumn{1}{|c|}{} & \multicolumn{1}{|c|}{} & \multicolumn{1}{|c|}{} & \multicolumn{1}{|c|}{}\\
\hline

\caption{Actions on simply connected positively curved manifolds}
\label{Tb:Examples-list}
\end{longtable}
\end{center}

\newpage

\begin{center}
\begin{longtable}{ |p{4.5cm}|p{4.5cm}|p{4.5cm}|}
\hline

$1.$

\begin{center}
\begin{tikzpicture}
\draw (-1.9,1) -- (-1,0.3) node[midway, above] {$2$};
\draw (-1,0.3) -- (0.1,0.3) node[midway, above] {$k$};
\draw  (-1.9,-0.4) -- (-1,0.3) node[midway, below] {$2$};
\filldraw[white] (-1,0.3) circle (2pt);
\draw[black] (-1,0.3) circle (2pt);
\filldraw[black] (-1.9,1) circle (2pt); 
\filldraw[black] (0.1,0.3) circle (2pt);
\filldraw[black] (-1.9,-0.4) circle (2pt);
\filldraw[black] (-0.8,1.5) circle (0pt);
\filldraw[black] (-0.8,-1.4) circle (0pt) node[anchor=south] {$(k\neq 2)$};
\end{tikzpicture}
\end{center}\vspace{-0.5cm}

&
$2.$

\begin{center}
\begin{tikzpicture}
\draw (0,-0.2) -- (-0.78,1);
\filldraw[black] (-0.26,0.26) circle (0pt) node[anchor=east] {$k$};
\filldraw[black] (0,1.7) circle (0pt) node[anchor=north] {$2$};
\filldraw[black] (0,0.3) circle (0pt) node[anchor=south] {$2$};
\filldraw[black] (0.78,1) circle (2pt);
\filldraw[white] (-0.78,1) circle (2pt);
\draw[black] (-0.78,1) circle (2pt);
\filldraw[black] (0,-0.2) circle (2pt);
\draw (0,1) ellipse [x radius=22pt, y radius=6pt];
\filldraw[white] (-0.78,1) circle (2pt);
\draw[black] (-0.78,1) circle (2pt);
\end{tikzpicture}
\end{center}\vspace{-0.5cm}

&
$3.$

\begin{center}
\begin{tikzpicture}
\draw (0,-0.2) -- (-0.78,1);
\filldraw[black] (-0.26,0.26) circle (0pt) node[anchor=east] {$2$};
\filldraw[black] (0,1.7) circle (0pt) node[anchor=north] {$k$};
\filldraw[black] (0,0.3) circle (0pt) node[anchor=south] {$2$};
\filldraw[black] (0.78,1) circle (2pt);
\filldraw[white] (-0.78,1) circle (2pt);
\draw[black] (-0.78,1) circle (2pt);
\filldraw[black] (0,-0.2) circle (2pt);
\draw (0,1) ellipse [x radius=22pt, y radius=6pt];
\filldraw[white] (-0.78,1) circle (2pt);
\draw[black] (-0.78,1) circle (2pt);
\filldraw[black] (0,-1.2) circle (0pt) node[anchor=south] {$(k\neq 3)$};
\end{tikzpicture}
\end{center}\vspace{-0.5cm}\\  
\hline

$4.$

\begin{center}
\begin{tikzpicture}
\draw (0,-0.2) -- (-0.78,1);
\filldraw[black] (-0.26,0.26) circle (0pt) node[anchor=east] {$3$};
\filldraw[black] (0,1.7) circle (0pt) node[anchor=north] {$2$};
\filldraw[black] (0,0.3) circle (0pt) node[anchor=south] {$3$};
\filldraw[black] (0.78,1) circle (2pt);
\filldraw[white] (-0.78,1) circle (2pt);
\draw[black] (-0.78,1) circle (2pt);
\filldraw[black] (0,-0.2) circle (2pt);
\draw (0,1) ellipse [x radius=22pt, y radius=6pt];
\filldraw[white] (-0.78,1) circle (2pt);
\draw[black] (-0.78,1) circle (2pt);
\end{tikzpicture}
\end{center}\vspace{-0.4cm}

&
$5.$

\begin{center}
\vspace{-0.6cm}
\begin{tikzpicture}
\draw (0,-0.2) -- (-0.78,1);
\filldraw[black] (-0.26,0.26) circle (0pt) node[anchor=east] {$k$};
\filldraw[black] (0,1.7) circle (0pt) node[anchor=north] {$2$};
\filldraw[black] (0,0.3) circle (0pt) node[anchor=south] {$\ell$};
\filldraw[black] (0.78,1) circle (2pt);
\filldraw[white] (-0.78,1) circle (2pt);
\draw[black] (-0.78,1) circle (2pt);
\filldraw[black] (0,-0.2) circle (2pt);
\draw (0,1) ellipse [x radius=22pt, y radius=6pt];
\filldraw[white] (-0.78,1) circle (2pt);
\draw[black] (-0.78,1) circle (2pt);
\filldraw[black] (0,-1.1) circle (0pt) node[anchor=south] {$\{k,\ell\}=\{3,4\}$ or $\{3,5\}$};
\end{tikzpicture}
\end{center}\vspace{-0.4cm}

&
$6$

\begin{center}
\begin{tikzpicture}
\draw (-0.8,0) -- (0.8,0);
\draw (-0.8,0) -- (0,1.6);
\draw (0.8,0) -- (0,1.6);
\draw (-0.8,0) -- (0,0.63);
\draw (0.8,0) -- (0,0.63);
\draw (0,0.63) -- (0,1.6);
\filldraw[black] (0,1.6) circle (2pt);
\filldraw[white](-0.8,0) circle (2pt);
\draw[black](-0.8,0) circle (2pt);
\filldraw[white] (0.8,0) circle (2pt);
\draw[black] (0.8,0) circle (2pt);
\filldraw[white] (0,0.63) circle (2pt);
\draw[black] (0,0.63) circle (2pt);
\filldraw[black] (0,1.8) circle (0pt);
\filldraw[black] (0.56,1.22) circle (0pt) node[anchor=north] {$2$};
\filldraw[black] (0.3,0.7) circle (0pt) node[anchor=north] {$2$};
\filldraw[black] (0.1,1.22) circle (0pt) node[anchor=north] {$3$};
\filldraw[black] (-0.3,0.7) circle (0pt) node[anchor=north] {$3$};
\filldraw[black] (-0.56,1.22) circle (0pt) node[anchor=north] {$2$};
\filldraw[black] (0,0.08) circle (0pt) node[anchor=north] {$3$};
\draw[black] (0,2) circle (0pt);
\end{tikzpicture}
\end{center}\vspace{-0.4cm}\\
\hline

\caption{Weighted graphs for which we do not know corresponding quotients $M/G$.}
\label{Tb:unknown-examples}
\end{longtable}
\end{center}

\bibliographystyle{plain}		

\begin{thebibliography}{1}

\bibitem{Br72}
G. E. Bredon, \emph{Introduction to Compact Transformation Groups}, Academic Press, New York, 1972.

\bibitem{BBI01}
D. Burago, Y. Burago, S. Ivanov, \emph{A Course in Metric Geometry}, American Mathematical Society, Providence, 2001.


\bibitem{BZ03}
G. Burde, H. Zieschang, \emph{Knots}, 
De Gruyter Stud. Math., vol 5, New York, 2003.




\bibitem{GGS14}
F. Galaz-Garcia, C. Searle, \emph{Nonnegatively curved  5-manifolds with almost maximal symmetry rank}, 
Geom. Topol. {\bf 18} (2014), 1397--1435.

\bibitem{Gro02}
K. Grove, \emph{Geometry of, and via, symmetries}, Conformal, Papers from the 30th John H. Barrett Memorial Lectures delivered at the University of Tennessee, Knoxville, TN. Univ. Lecture Ser. 27, American Mathematical Society, Providence, RI (2002).


\bibitem{GH87}
K. Grove. S. Halperin, \emph{Dupin hypersurfaces, group actions and the double mapping cylinder},
 J. Differential Geom. {\bf26} (1987), 429--459.

\bibitem{GM}
K. Grove, S. Markvorsen, \emph{New extremal problems for the Riemannian recognition program via Alexandrov geometry}, J. Amer. Math. Soc. {\bf8} (1995), 1--28.

\bibitem{GW14}
K. Grove, B. Wilking, \emph{A knot characterization and $1$-connected non-negatively curved $4$-manifolds with circle symmetry}, 
Geometry and Topology {\bf 18} (2014), no. 5, 3091--3110.

\bibitem{GWY19}
K. Grove, B. Wilking, J. Yeager, \emph{Almost non-negative curvature and rational ellipticity in cohomogeneity two}, 
Ann. Inst. Fourier {\bf 69} (2019), no. 7, 2921--2939.

\bibitem{GWZ08}
K. Grove, B. Wilking, W. Ziller, \emph{Positively curved cohomogeneity one manifolds and $3$-Sasakian geometry}, 
J. Differ. Geom. {\bf 78} (2008), no. 1, 33--111.

\bibitem{GZ12}
K. Grove, W. Ziller, \emph{Polar manifolds and actions},
J. Fixed Point Theory Appl. {\bf 11} (2012), 279--313.

\bibitem{HK89}
W. Y.  Hsiang, B. Kleiner, \emph{On the topology of positively curved $4$-manifolds with symmetry}, 
J. Differ. Geom. {\bf 29} (1989), no. 3, 615--621.


\bibitem{KR24}
E. Khalili Samani, M. Radeschi, \emph{Rational ellipticity of $G$-manifolds from their quotients},
Compositio Mathematica (2024), to appear.

\bibitem{Lyt10}
A. Lytchak, \emph{Geometric resolution of singular Riemannian foliations}, 
Geom. Dedicata {\bf 149} (2010), 379--395.


\bibitem{LT10}
A. Lytchak, G. Thorbergsson, \emph{Curvature explosion in quotients and applications}, 
J. Differ. Geom. {\bf 85} (2010), no. 1, 117--140.

\bibitem{MGS22}
J. McGowan, C. Searle, \emph{How Tightly Can You Fold a Sphere?},
Differential Geometry and its Applications {\bf22} (2005), 81--104.

\bibitem{Men21}
R.A.E. Mendes, \emph{Lifting isometries of orbit spaces}, 
Bull. Lond. Math. Soc. {\bf53} (2021), no. 6, 1621--1626.

\bibitem{MR20}
R. Mendes, M. Radeschi, \emph{Laplacian algebras, manifolds submetries and their inverse invariant theory problem}, 
Geom. Func. Anal. {\bf 30} (2020), 536--573.


\bibitem{Pet16}
P. Petersen, \emph{Riemannian geometry, Third Edition}, 
Graduate Texts in Mathematics, Springer, 2016.

\bibitem{Rol03}
D. Rolfsen, \emph{Knots and Links}, United Kingdom: AMS Chelsea Pub., 2003.

\bibitem{Ser53}
J. P. Serre, \emph{Groupes d'homotopie et classes de grupes ab\'eliens},
Ann. of Math. {\bf 58} (1953), 258--294.

\bibitem{Str94}
E. Straume, \emph{On the invariant theory and geometry of compact linear groups of cohomogeneity $\leq 3$}, Differential Geometry and its Applications {\bf 4} (1994) 1--23.

\bibitem{Wil04}
B. Wilking, \emph{Positively curved manifolds with symmetry}, Ann. of Math. {\bf163} (2006), no. 2, 607--668.

\end{thebibliography}

\end{document}